\documentclass[11pt,a4paper]{amsart}

\usepackage[margin=1.1in]{geometry}
\usepackage{microtype}
\usepackage{hyphenat}
\usepackage{parskip}

\usepackage{amsmath,amssymb,amsthm,mathtools}
\usepackage{commath}
\usepackage{relsize}

\usepackage{color}
\usepackage{graphicx}
\usepackage{array}
\usepackage{subfig}
\usepackage{wrapfig}
\usepackage{float}
\usepackage[normalem]{ulem}
\usepackage{outlines}
\usepackage{enumitem} 
\usepackage{comment}
\usepackage{kantlipsum}

\usepackage{hyperref}
\usepackage{cleveref} 
\hypersetup{colorlinks=true, linkcolor=blue, citecolor=blue}

\crefname{enumi}{condition}{conditions}
\Crefname{enumi}{Condition}{Conditions}

\newcommand{\CC}{\mathbb{C}}
\newcommand{\ZZ}{\mathbb{Z}}
\newcommand{\Q}{\mathbb{Q}}
\newcommand{\Qbar}{\overline{\mathbb{Q}}}

\mathtoolsset{showonlyrefs=true}

\theoremstyle{plain}
\newtheorem{theorem}{Theorem}[section]
\newtheorem{corollary}[theorem]{Corollary}
\newtheorem{proposition}[theorem]{Proposition}
\newtheorem{lemma}[theorem]{Lemma}
\newtheorem{defn}[theorem]{Definition}

\theoremstyle{definition}
\newtheorem{remark}[theorem]{Remark}
\newtheorem{example}[theorem]{Example}

\DeclareMathOperator{\ddiv}{div}      

\DeclareMathOperator{\disc}{disc}

\DeclareMathOperator{\ord}{ord}

\title{\bfseries Weighted Fruit Diophantine Equations and Hyperelliptic Curves}

\author{Jewel Mahajan}
\address{Department of Mathematical Sciences, Indian Institute of Science Education and Research (IISER) Berhampur, Ganjam, Odisha 760003, India}
\email{jewelmahajan421@gmail.com}

\author{Apeksha Sanghi}
\address{Department of Mathematics, Indian Institute of Science (IISc), CV Raman Road, Bengaluru 560012, Karnataka, India}
\email{apekshasanghi93@gmail.com} 


\begin{document}
\begin{abstract}
We study the \textit{weighted fruit Diophantine equation} $  ax^{d} - c\bigl(m^{2}y^{2}+n^{2}z^{2}\bigr) + xyz - b = 0$, generalising the results established in \cite{MajumdarSury,VaishyaSharma,PrakashChakraborty}.
Subject to specific hypotheses on the parameters, our main result shows that for any prime $l \equiv 3 \pmod 4$ and $b = a (2 c m n)^{d} - l\, c^{s}t^{2q}$, the above equation has no integer solutions except for certain residue classes of $x$ modulo $4l$. An analogous result also holds when $l$ is replaced by an odd power of $l$ in the definition of $b$. We prove some insolvability results for $l=-1$. By applying the main result to the small values of $l$, such as $l \in \{3, 7, 11, 19\}$, we explicitly determine the exceptional residue classes outside of which the equation has no solutions. In particular, for $l = 3$, this yields complete insolvability, and weakening these hypotheses still yields non-existence results, though with specific coprimality restrictions on any possible solutions. We also consider a more general variant of the above Diophantine equation and provide some insolvability results. Subsequently, we establish bounds for the positive solutions of the aforementioned equation. Finally, by associating a family of hyperelliptic curves with the equation under consideration and applying Grant's analogue of the Nagell--Lutz theorem \cite{Grant}, we translate these insolvability results into results about their rational torsion points.
\end{abstract}

\maketitle

\section{Introduction}\label{sec:intro}
Diophantine equations have a remarkably rich literature in number theory. A significant part of this research focuses on determining whether a given Diophantine equation admits integer solutions. This area of study also shares deep connections with algebraic geometry, particularly concerning integral points on elliptic (resp., hyperelliptic) curves via the Nagell--Lutz theorem (resp., Grant's theorem).

Consider the Diophantine equation
\begin{equation}\label{eq:fruit}
  y^2 - xyz + z^2 = x^3 - 5.
\end{equation}
Majumdar and Sury \cite{MajumdarSury} proved that~\eqref{eq:fruit} has no integer solution. They also showed that the associated elliptic curve $y^2 - mxy = x^3 - (m^2 + 5)$ has no integral point for any integer $m$. Equations of a similar form, $ax^3 + by + c - xyz = 0$, were previously investigated by Luca and Togb\'e \cite{LucaTogbe2008, Togbe2009}, Subburam and Thangadurai \cite{SubburamThangadurai2014, SubburamThangadurai2015}, and Subburam and Togb\'e \cite{SubburamTogbe2017}.

In an effort to generalise existing results in the literature, we consider the Diophantine equation
\begin{equation}\label{eq:main_eq}
  ax^{d} - c\bigl(m^{2}y^{2} + n^{2}z^{2}\bigr) + xyz - b = 0.
\end{equation}
Extending the work of Majumdar and Sury, who studied~\eqref{eq:main_eq} for the specific case where $a=c=m=n=1$, $b=5$, and $d=3$, Vaishya and Sharma \cite{VaishyaSharma} proved the insolvability of~\eqref{eq:main_eq} when $a \equiv 1 \pmod{12}$, $c = m = n = 1$, $b = 8a - 3$, and $d=3$. This yielded an infinite family of elliptic curves with torsion-free Mordell--Weil group over $\Q$ (see \cite[Corollary~1.4]{VaishyaSharma}).

Along the same line, further extending the work of Vaishya and Sharma, Prakash and Chakraborty \cite{PrakashChakraborty} established the insolvability of~\eqref{eq:main_eq} for the case when $a \equiv 1 \pmod{12}$, $c = m = n = 1$, $b = a2^d - 3$, and $d$ is an odd multiple of $3$. By applying Grant's hyperelliptic analogue of the Nagell--Lutz theorem \cite{Grant}, they obtained infinitely many hyperelliptic curves with torsion-free Mordell--Weil groups. This is consistent with Faltings' finiteness theorem \cite{Faltings1983} and Bhargava's result \cite{Bhargava2013} that most hyperelliptic curves over $\Q$ have no rational points. We remark that \cite[Theorem~3.3]{PrakashChakraborty} and \cite[Theorem~1.2]{VaishyaSharma} are valid only for $a = 1$, genus $g=1$, and when $m$ specified therein is even (see Remark~\ref{rem:pc-correction}).

This article aims to investigate the insolvability of~\eqref{eq:main_eq} under some hypotheses on its parameters. The remainder of this article is organised as follows. In Section~\ref{sec:prelim}, we define the sets $\mathcal{N}_1(l)$ and $\mathcal{N}_2(l)$ (see Definition~\ref{N_1 definition}) and establish several propositions and lemmas for later use. In Section~\ref{sec:main}, assuming the parameters $a,c,d,m,n,q,s,t$ satisfy the hypotheses of Theorem~\ref{thm:main_general_l}, we establish the insolvability of the generalised Diophantine equation~\eqref{eq:main_eq} for a prime $l \equiv 3 \pmod{4}$ and parameter $b = a (2 c m n)^{d} - l c^{s} t^{2q}$. Our proof generalises methods from previous work by analysing the residue classes of $x$ modulo $4l$, with the careful observation made in Theorem~\ref{thm:main_general_l}\ref{cond:4}. A significant portion of the proof of Theorem~\ref{thm:main_general_l} remains valid even if $l$ is replaced by an odd power of $l$ in the definition of $b$. This yields a similar insolvability result, with the only exception that integer solutions $(x,y,z)$ with $x \equiv 2cmn \pmod l$ cannot be ruled out (see Corollary~\ref{thm:main_general_l_k}). Subsequently, we apply Theorem~\ref{thm:main_general_l} to some small values of $l$ in Section~\ref{sec:smallprimes}, such as $l \in \{ 3, 7, 11, 19\}$, and explicitly determine the exceptional residue classes outside of which~\eqref{eq:main_eq} has no solutions. In particular, for $l = 3$, this yields complete insolvability (see Theorem~\ref{thm:main3_gen}), while for the small primes $l = 7, 11$, and $ 19$, we list the set of residues modulo $28, 44$, and $76$, respectively, and exhibit solutions in the exceptional classes, showing that Theorems~\ref{thm:main7_gen},~\ref{thm:main11_gen} and~\ref{thm:main19_gen} are sharp. Similar ideas also yield some insolvability results for $l=-1$.

In subsequent sections, we investigate~\eqref{eq:main_eq} under various hypotheses, establishing a connection between an associated family of hyperelliptic curves and the insolvability of~\eqref{eq:main_eq}. In Section~\ref{sec:coprime}, we explore a variant of Theorem~\ref{thm:main3_gen} under weakened hypotheses. We observe that non-existence results continue to hold, though with specific coprimality restrictions on any potential solutions (see Theorem~\ref{thm:coprime_3}). Section~\ref{sec:asymmetric} considers a more general variant of~\eqref{eq:main_eq}, where the coefficients of $y^2$ and $z^2$ are independent, specifically allowing them to be squarefree, and we also provide some insolvability results depending on the residue of $a$ modulo $4$. In Section~\ref{sec:positive solution}, we focus on positive solutions to~\eqref{eq:main_eq}. In particular, under the hypothesis on the parameters of~\eqref{eq:main_eq}, we show that there are only finitely many positive integer solutions $(x,y,z)$ satisfying $x \le 2cmn$. Finally, in Section~\ref{sec:hyperelliptic}, we associate a family of hyperelliptic curves with~\eqref{eq:main_eq} and analyse their rational torsion points. As a consequence, Theorem~\ref{thm:main3_gen} yields a family of elliptic curves with torsion-free Mordell--Weil groups (see Theorem~\ref{thm:general-3c-hyperelliptic}). We believe that this framework of connecting rational torsion points to the solution of certain Diophantine equations could possibly be extended to superelliptic curves, assuming an analogue to the Grant or Nagell-Lutz theorem exists in that setting. We leave this direction for future work.

\section{Preliminaries}\label{sec:prelim}
Throughout this article, $\bigl(\frac{\cdot}{\cdot}\bigr)$ denotes the Jacobi symbol. Note that this reduces to the Legendre symbol when the denominator is an odd prime $p$. This section establishes the preliminary definitions, lemmas, and propositions necessary for subsequent sections.
\begin{defn}\label{N_1 definition}
Let $l$ be a positive integer. We define the sets $\mathcal{N}_1(l)$ and $\mathcal{N}_2(l)$ by
\[
  \mathcal{N}_1(l) = \bigg\{ x \in \ZZ \;\bigg|\; 1 \le x < 4l, \, \gcd(x, 4l) = 1, \, \left(\frac{l}{x}\right) = 1 \bigg\},
\]
\[
  \mathcal{N}_2(l) = \bigg\{ x \in \ZZ \;\bigg|\; 1 \le x < 4l, \, \gcd(x, 4l) = 1, \, \left(\frac{l}{x}\right) = -1 \bigg\}.
\]
\end{defn}

\begin{remark}
Following the convention that an empty product evaluates to $1$, we always have $1 \in \mathcal{N}_1(l)$.
\end{remark}

\begin{proposition}\label{prop:leg_general_3mod4}
Let $l$ be a prime such that $l \equiv 3 \pmod 4$. For any odd prime $p \neq l$, the Legendre symbol $\left(\frac{l}{p}\right)$ is determined by the congruence class of $p$ modulo $4l$. Specifically, $\left(\frac{l}{p}\right) = 1$ if and only if $p$ satisfies one of the following conditions:
\begin{enumerate}[label=(\alph*)]
\item $p \equiv 1 \pmod 4$ and $p \equiv a \pmod l$, where $a$ is a quadratic residue modulo $l$;
    \item $p \equiv 3 \pmod 4$ and $p \equiv b \pmod l$, where $b$ is a quadratic nonresidue modulo $l$.
\end{enumerate}
By the Chinese Remainder Theorem (CRT), there are exactly $l-1$ residue classes modulo $4l$ for which $\left(\frac{l}{p}\right) = 1$, and $\left(\frac{l}{p}\right) = -1$ for the remaining $l-1$ residue classes.
\end{proposition}

\begin{proof}
By the law of quadratic reciprocity (see \cite[Theorem~9.9]{Burton}), for an odd prime $p \neq l$ with $l \equiv 3 \pmod{4}$, we have
\[
\left(\frac{l}{p}\right) = \left(\frac{p}{l}\right) (-1)^{\frac{l-1}{2} \cdot \frac{p-1}{2}}= \left(\frac{p}{l}\right) (-1)^{\frac{p-1}{2}}.
\]
We analyse the following cases for $p$ modulo $4$ to determine the sign of the second factor.
\begin{enumerate}[label=(\alph*)]
\item If $p \equiv 1 \pmod 4$, then $(-1)^{\frac{p-1}{2}} = 1$. Thus, $\left(\frac{l}{p}\right) = \left(\frac{p}{l}\right)$. For this to equal $1$, $p$ must be a quadratic residue modulo $l$.
    \item If $p \equiv 3 \pmod 4$, then $(-1)^{\frac{p-1}{2}} = -1$. Thus, $\left(\frac{l}{p}\right) = -\left(\frac{p}{l}\right)$. For this to equal $1$, $p$ must be a quadratic nonresidue modulo $l$.
\end{enumerate}
Since $\gcd(4, l) = 1$, the CRT guarantees that each possible combination mentioned above produces a unique residue class modulo $4l$. These computations yield the exact residue classes for which $\left(\frac{l}{p}\right) = 1$, and the remaining residue classes yield $\left(\frac{l}{p}\right) = -1$.
\end{proof}
Using Proposition~\ref{prop:leg_general_3mod4}, we now determine the residue classes $p$ modulo $4l$ for which $\left(\frac{l}{p}\right)=\pm1$ for each $l \in \{3,7,11,19\}$.
\begin{proposition}\label{prop:leg3}
For an odd prime $p \neq 3$, we have
\[
\left(\frac{3}{p}\right) = 
\begin{cases}
    1 & \text{if } p \equiv 1, 11 \pmod{12}, \\
    -1 & \text{if } p \equiv 5, 7 \pmod{12}.
\end{cases}
\]
\end{proposition}
\begin{proof}
Applying Proposition~\ref{prop:leg_general_3mod4} with $l=3$, we have $\left(\frac{3}{p}\right) = 1$ if either $p \equiv 1 \pmod 4$ with $p \equiv 1 \pmod 3$, or $p \equiv 3 \pmod 4$ with $p \equiv 2 \pmod 3$. Solving these simultaneous congruences via the CRT yields $p \equiv 1, 11 \pmod{12}$, with the remaining residue classes yielding $-1$.
\end{proof}
\begin{proposition}\label{prop:leg7}
For an odd prime $p \neq 7$, we have
\[
\left(\frac{7}{p}\right) = 
\begin{cases}
    1 & \text{if } p \equiv 1, 3, 9, 19, 25, 27 \pmod{28}, \\
    -1 & \text{if } p \equiv 5, 11, 13, 15, 17, 23 \pmod{28}.
\end{cases}
\]\end{proposition}
\begin{proof}
Applying Proposition~\ref{prop:leg_general_3mod4} with $l=7$, we have $\left(\frac{7}{p}\right) = 1$ if either $p \equiv 1 \pmod 4$ with $p \equiv 1, 2, 4 \pmod 7$, or $p \equiv 3 \pmod 4$ with $p \equiv 3, 5, 6 \pmod 7$. Solving these simultaneous congruences via the CRT yields $p \equiv 1, 3, 9, 19, 25, 27 \pmod{28}$, with the remaining residue classes yielding $-1$.
\end{proof}
\begin{proposition}\label{prop:leg11}
For an odd prime $p \neq 11$, we have
\[
\left(\frac{11}{p}\right) = 
\begin{cases}
    1 & \text{if } p \equiv 1, 5, 7, 9, 19, 25, 35, 37, 39, 43 \pmod{44}, \\
    -1 & \text{if } p \equiv 3, 13, 15, 17, 21, 23, 27, 29, 31, 41 \pmod{44}.
\end{cases}
\]\end{proposition}
\begin{proof}
Applying Proposition~\ref{prop:leg_general_3mod4} with $l=11$, we have $\left(\frac{11}{p}\right) = 1$ if either $p \equiv 1 \pmod 4$ with $p \equiv 1, 3, 4, 5, 9 \pmod{11}$, or $p \equiv 3 \pmod 4$ with $p \equiv 2, 6, 7, 8, 10 \pmod{11}$. Solving these simultaneous congruences via the CRT yields $p \equiv 1, 5, 7, 9, 19, 25, 35, 37, 39, 43 \pmod{44}$, with the remaining residue classes yielding $-1$.
\end{proof}
\begin{proposition}\label{prop:leg19}
For an odd prime $p \neq 19$, we have
\[
\left(\frac{19}{p}\right) = 
\begin{cases} 
    1 & \text{if } p \equiv 1, 3, 5, 9, 15, 17, 25, 27, 31, \\
      & \quad 45, 49, 51, 59, 61, 67, 71, 73, 75 \pmod{76}, \\ 
    -1 & \text{if } p \equiv 7, 11, 13, 21, 23, 29, 33, 35, 37, \\
       & \quad 39, 41, 43, 47, 53, 55, 63, 65, 69 \pmod{76}. 
\end{cases}
\]
\end{proposition}
\begin{proof}
Applying Proposition~\ref{prop:leg_general_3mod4} with $l=19$, we have $\left(\frac{19}{p}\right) = 1$ if either $p \equiv 1 \pmod 4$ with $p \equiv 1, 4, 5, 6, 7, 9, \allowbreak 11, 16, 17 \pmod{19}$, or $p \equiv 3 \pmod 4$ with $p \equiv 2, 3, 8, 10, 12, \allowbreak 13, 14, 15, 18 \pmod{19}$. Solving these simultaneous congruences via the CRT yields $p \equiv 1, 3, 5, 9, 15, 17, \allowbreak 25, 27, 31, 45, 49, 51, \allowbreak 59, 61, 67, 71, 73, 75 \pmod{76}$, with the remaining residue classes yielding $-1$.
\end{proof}
\begin{lemma}[{\cite[Theorem~1.4.9]{Cohen}}]\label{lemma:jacobi_periodicity}
Let $a$ be a nonzero integer. For any odd positive integers $m$ and $n$ such that $m \equiv n \pmod{4|a|}$, we have
\[
    \left(\frac{a}{m}\right) = \left(\frac{a}{n}\right).
\]
\end{lemma}
\begin{lemma}\label{lem:sum_div}
Let $l$ and $d$ be positive integers such that $l \mid d$. If $x \equiv A \pmod l$ for some integer $A$, then $\sum_{i=0}^{d-1} x^{d-1-i}A^i \equiv 0 \pmod l$.
\end{lemma}
\begin{proof}Since $l\mid d$, we have
\begin{equation}
    \sum_{i=0}^{d-1} x^{d-1-i}A^i \equiv\sum_{i=0}^{d-1} A^{d-1-i}A^i \equiv \sum_{i=0}^{d-1} A^{d-1} \equiv d A^{d-1}\equiv0 \pmod l.
\end{equation}
\end{proof}
\begin{lemma}\label{lem:t_odd}
    Let $l$ be a prime, and let $t \in \ZZ$ be such that every prime factor of $t$ is congruent to $r$ for some $r \in \mathcal{N}_1(l)$ (see Definition~\ref{N_1 definition}). Then $t$ is odd and $l \nmid t$. 
\end{lemma}
\begin{proof}
Let $p$ be a prime factor of $t$. By assumption, $p \equiv r \pmod{4l}$. Since $r \in \mathcal{N}_1(l)$ is odd, $p$ must also be odd. Because every prime factor of $t$ is odd, it follows that $t$ is odd.

Also, since $r \in \mathcal{N}_1(l)$, we have $\gcd(r, 4l) = 1$, which means that $l\nmid r$. Consequently, $l \nmid p$ (i.e., $p \neq l$). Because $l$ is a prime and it is different from all the prime factors of $t$, it follows that $l \nmid t$.
\end{proof}
\begin{defn}\label{def:discriminant}
A polynomial $h(x) \in \ZZ[x]$ is \textit{squarefree} if it is not
divisible by $g(x)^2$ for any non-constant polynomial $g(x)$. Equivalently, $h(x)$ is squarefree if it has no repeated roots in $\Qbar$, the algebraic closure of $\Q$.

Let $h(x) \in \ZZ[x]$ be monic of degree $n \ge 1$, and let
$\alpha_1, \dots, \alpha_n \in \CC$ denote its roots, repeated as many times according to multiplicity, so that $  h(x) = \prod_{i=1}^{n} (x - \alpha_i) \in \CC[x]$. The \textit{discriminant} of $h$ is defined as
\[
  \disc(h) \;=\; \prod_{1 \le i < j \le n} (\alpha_i - \alpha_j)^2 .
\]
\end{defn}
\begin{remark}\label{discriminant nonzero}
The product above is equal to $P(e_1,\ldots,e_n)$ for some $P\in \ZZ[X_{1},\dots,X_n]$,  where $e_1, \dots, e_n$  are the elementary symmetric polynomials in $\alpha_1, \dots, \alpha_n$. By Vieta's formulas, each
$e_k = e_k(\alpha_1, \dots, \alpha_n)$ equals a
coefficient of $h$ (up to sign), and is therefore an integer. Consequently, $\disc(h) \in \ZZ$. Furthermore, $h$ is squarefree if and only if $\disc(h) \ne 0$.

\end{remark}

\begin{lemma}[{\cite[Theorem~4]{Greenfield1984}}]\label{lem:trinomial_discriminant}
For an integer $d > 2$, the discriminant of the trinomial $f(X) = X^d + \lambda X^2 + \mu$ is given by
\begin{equation}\label{eq:discriminant_formula}
\disc(f) = (-1)^{d(d-1)/2} \mu \left( d^N \mu^{N-K} - (-1)^N (d-2)^{N-K} 2^K \lambda^N \right)^{\tilde{g}},
\end{equation}
where $\tilde{g} = \gcd(d, 2)$, $N = d/\tilde{g}$, and $K = 2/\tilde{g}$. 

In particular, when $d>2$ is odd, the discriminant simplifies to
\[
\disc(f) = (-1)^{d(d-1)/2} \mu \big( d^d \mu^{d-2} + 4(d-2)^{d-2} \lambda^d \big).
\]
\end{lemma}

\section{Insolvability for primes \texorpdfstring{$l \equiv 3 \pmod 4$}{l = 3 mod 4}}\label{sec:main}

We now prove one of the main results of this article, the key insolvability result for~\eqref{eq:main_eq} under some assumptions on the parameters $a,c,d,m,n,t,q,s$ and $b$.
\begin{theorem}\label{thm:main_general_l}
Let $l$ be a prime such that $l \equiv 3 \pmod 4$. Let $a, c, m, n, t \in \ZZ$ and $d, q, s \in \ZZ_{>0}$ be such that $d, m, n$, and $s$ are odd integers with $l \mid d$. Assume $a \equiv c \equiv 1 \pmod{4l}$ and $mn \equiv 1 \pmod{l}$, and suppose that every prime factor of $cmt$ is congruent to $r$ modulo $4l$ for some $r \in \mathcal{N}_1(l)$. Set $b = a (2 c m n)^{d} - l c^{s} t^{2q}.$ Then~\eqref{eq:main_eq} has no integer solutions $(x,y,z) $ for which $x$ satisfies any of the following conditions:
\begin{enumerate}[label=(\alph*)]
    \item $x$ is even; \label{cond:1}
    \item $x \equiv 3 \pmod 4$; \label{cond:2}
    \item $x \equiv \pm 2cmn \pmod l$;\label{cond:3}
    \item $x - 2cmn \equiv v \pmod{4l}$ for some $v \in \mathcal{N}_2(l) $\label{cond:4}.
\end{enumerate}
\end{theorem}
\begin{proof}
Since $d \ge 3$, $c \equiv 1 \pmod 4$, and $t$ is odd (by Lemma~\ref{lem:t_odd}), it follows that $2^d \equiv 0 \pmod 4$ and $c^s \equiv t^{2q} \equiv 1 \pmod 4$. Therefore,
\begin{equation}\label{eq:g3c-prelim}
b \equiv -lc^{s}t^{2q} \equiv -l \equiv -3\equiv 1 \pmod{4}.
\end{equation}
We divide our analysis into two cases based on the parity of $x$.

\medskip
\noindent\textbf{Case 1}: \textit{$x$ is even.}\quad
Let $x = 2\beta$. Multiplying~\eqref{eq:main_eq} by $cm^{2}$ and completing the square via the substitution $U := cm^{2}y - \beta z$ yields
\begin{equation}\label{eq:g3c-Ceven}
U^{2} - (\beta^{2}-c^{2}m^{2}n^{2})z^{2} = cm^{2}\bigl(a(2\beta)^{d}-b\bigr).
\end{equation}
Using the definition of $b$ and the fact that $4 \mid 2^{d}$, the right-hand side of~\eqref{eq:g3c-Ceven} becomes
\[
cm^{2}\bigl(a(2\beta)^{d}-b\bigr) = acm^{2}\bigl((2\beta)^{d}-(2cmn)^{d}\bigr) + lc^{s+1}m^{2}t^{2q} \equiv lc^{s+1}m^{2}t^{2q} \pmod{4}.
\]
Because $m$ and $n$ are odd, $m^2 \equiv n^2 \equiv c^{2}m^{2}n^{2} \equiv 1 \pmod{4}$. Since $c \equiv t^2\equiv 1 \pmod 4$, the right-hand side of~\eqref{eq:g3c-Ceven} is congruent to $3 \pmod 4$. 
\begin{itemize}
    \item If $\beta$ is even, then $\beta^{2} - c^{2}m^{2}n^{2} \equiv -1 \pmod 4$. Thus~\eqref{eq:g3c-Ceven} reduces to $U^{2} + z^{2} \equiv 3 \pmod 4$, which is impossible, since a sum of two squares must be congruent to $0, 1$, or $2 \pmod 4$.
    \item If $\beta$ is odd, then $\beta^{2} - c^{2}m^{2}n^{2} \equiv 0 \pmod 4$. Thus~\eqref{eq:g3c-Ceven} reduces to $U^{2} \equiv 3 \pmod 4$, which is also impossible, since $3$ is not a quadratic residue modulo $4$.
\end{itemize}
This proves~\ref{cond:1}.

\medskip
\noindent\textbf{Case 2}: \textit{$x$ is odd.}\quad
Reducing~\eqref{eq:main_eq} modulo $2$, noting that $a,c,m,n,x$, and $b$ are all odd, gives
\[
y^{2} + z^{2} - yz \equiv a-b\equiv 0 \pmod 2.
\]
This forces $y \equiv z \equiv 0 \pmod 2$. Write $y = 2Y$ and $z = 2Z$. Multiplying~\eqref{eq:main_eq} by $cm^{2}$ and completing the square via the substitution $w := 2cm^{2}Y - x Z$ yields
\begin{equation}\label{eq:g3c-Codd}
w^{2} - (x^{2}-4c^{2}m^{2}n^{2})Z^{2} = acm^{2}\bigl(x^{d}-(2cmn)^{d}\bigr) + lc^{s+1}m^{2}t^{2q}.
\end{equation}
The left-hand side of~\eqref{eq:g3c-Codd} expands to $4cm^{2}(cm^{2}Y^{2}-x YZ+cn^{2}Z^{2})$, which is divisible by $4$. Using $(2cmn)^{d} \equiv 0 \pmod 4$, $x^{d} \equiv x \pmod 4$ (since $d$ is odd), and $acm^{2} \equiv c^{s+1}m^{2}t^{2q} \equiv 1 \pmod 4$, the reduction of~\eqref{eq:g3c-Codd} modulo $4$ yields 
\begin{align}\label{xequiv1mod4}
0 \equiv x + 3 \pmod 4 \implies x \equiv 1 \pmod 4.
\end{align}
This immediately shows that~\eqref{eq:main_eq} admits no integer solutions when $x \equiv 3 \pmod 4$, thereby proving~\ref{cond:2}. For the remaining conditions, we assume $x \equiv 1 \pmod 4$ and analyse~\eqref{eq:g3c-Codd} in the following three subcases.

\medskip
\noindent\textbf{Subcase 2.1: $x \equiv 2cmn \pmod l$.}\quad
Assume, for the sake of contradiction, that~\eqref{eq:main_eq} admits an integer solution $(x, y, z)$ with $x \equiv 2cmn \pmod l$.

Therefore, we can write $x - 2cmn \equiv el \pmod{4l}$ for some $e \in \{-1,0, 1, 2 \}$. Since $x - 2cmn$ and $l$ are odd, we have $e \in \{-1, 1\}$. Thus, there exists an integer $\delta'$ such that $x = 2cmn + el + 4l\delta' = 2cmn + l(e + 4\delta')$. Setting $\delta = e + 4\delta'$, we have $x = 2cmn + l\delta$, where $\delta \equiv e \pmod 4$.

If $e = -1$, then $\delta \equiv -1 \pmod 4$. This implies $x = 2cmn + l(-1 + 4\delta')$. Since $c \equiv 1 \pmod 4$ and $mn $ is odd, we have $2cmn \equiv 2 \pmod 4$. This yields $x \equiv 2 - l \equiv  -1 \equiv 3 \pmod 4$, contradicting~\eqref{xequiv1mod4}.

Therefore, we must have $e = 1$, which means that $\delta \equiv 1 \pmod 4$ and $x - 2cmn=l\delta$. Evaluating~\eqref{eq:g3c-Codd} modulo $l$, we have $w^2 \equiv 0 \pmod l$. Therefore, $w = lw_1$ for some $w_1 \in \ZZ$. Substituting $w = lw_1$ and $x = 2cmn + l\delta$ back into~\eqref{eq:g3c-Codd} yields
\begin{equation}\label{eq:step2}
l^2 w_1^2 - l\delta(x + 2cmn)Z^2 = c m^{2}\big( a(x^d - (2cmn)^d) + l c^s t^{2q} \big).
\end{equation}
By the binomial expansion, $x^d - (2cmn)^d = (x - 2cmn)S = l\delta S$, where $S = \displaystyle\sum_{i=0}^{d-1} x^{d-1-i}(2cmn)^i$. Dividing~\eqref{eq:step2} by $l$ gives
\begin{equation}\label{eq:step3}
l w_1^2 - \delta(x + 2cmn)Z^2 = m^{2}(ac\delta S + c^{s+1} t^{2q}).
\end{equation}

As $l\mid d$, Lemma~\ref{lem:sum_div} yields $S \equiv 0 \pmod l$. Recall that $a \equiv c \equiv mn \equiv 1 \pmod l$, and $x \equiv 2cmn \pmod l$. Reducing~\eqref{eq:step3} modulo $l$ yields 
\begin{align}
    & -\delta(4cmn)Z^2 \equiv m^2c^{s+1} t^{2q} \pmod l \\
    \implies & -4\delta Z^2 \equiv t^{2q}m^2 \pmod l\\
    \implies & -4\delta n^2Z^2 \equiv t^{2q}m^2n^2\equiv t^{2q} \pmod l.
\end{align}

Observe that $l \nmid 2\delta n Z$; otherwise, we would have $l \mid t^{2q}$, which contradicts Lemma~\ref{lem:t_odd}. Consequently, $\left( \frac{\delta}{l} \right)$, $\left( \frac{n}{l} \right)$ and $\left( \frac{2Z}{l} \right)$ are nonzero. Since $t\neq 0 $ (by Lemma~\ref{lem:t_odd}) and $l \equiv 3 \pmod 4$, we have
\[1 = \left(\frac{t^{2q}}{l}\right) = \left( \frac{-4\delta n^2Z^2}{l} \right) = \left( \frac{-1}{l} \right) \left( \frac{\delta}{l} \right) \left( \frac{n}{l} \right)^2 \left( \frac{2Z}{l} \right)^2 = - \left( \frac{\delta}{l} \right).\]

Therefore, $\left(\frac{\delta}{l}\right) = -1$. We can write $\delta = \text{sgn}(\delta)|\delta|$, where $\text{sgn}(\cdot)$ denotes the signum function. Since $l \equiv 3 \pmod 4$, we have $\left(\frac{\text{sgn}(\delta)}{l}\right) = \text{sgn}(\delta)$. Using the multiplicative property of the Jacobi symbol, this gives
\begin{align}\label{sgn fn}
  -1 \;=\; \left(\frac{\delta}{l}\right) \;=\; \left(\frac{\text{sgn}(\delta)}{l}\right)\left(\frac{|\delta|}{l}\right) \;=\; \text{sgn}(\delta)\left(\frac{|\delta|}{l}\right).
\end{align}
Furthermore, the condition $\delta \equiv 1 \pmod 4$ implies that $|\delta| \equiv \text{sgn}(\delta) \pmod 4$. We now apply the generalised law of quadratic reciprocity (see \cite[p.~192]{Burton}) to the positive odd integers $l$ and $|\delta|$. Since $l \equiv 3 \pmod 4$, the reciprocity factor simplifies to $(-1)^{\frac{l-1}{2} \frac{|\delta|-1}{2}} = (-1)^{\frac{|\delta|-1}{2}} = \text{sgn}(\delta).$ 
Therefore, by~\eqref{sgn fn}, we have
\[
  \left(\frac{l}{|\delta|}\right) \;=\; (-1)^{\frac{l-1}{2} \frac{|\delta|-1}{2}} \left(\frac{|\delta|}{l}\right)
  \;=\;  \text{sgn}(\delta)\left(\frac{|\delta|}{l}\right) \;=\; -1.
\]
Let $\gamma \in [1,4l]\cap\ZZ$ satisfy $|\delta| \equiv \gamma \pmod{4l}$. Since $l \nmid |\delta|$ and $|\delta|$ is odd, we have $\gcd(|\delta|, 4l) = 1$. Combined with $\left(\frac{l}{|\delta|}\right) = -1$, Lemma~\ref{lemma:jacobi_periodicity} guarantees $\gamma \in \mathcal{N}_2(l)$. Since the Jacobi symbol is multiplicative and $\mathcal{N}_1(l)$ is
closed under multiplication modulo $4l$, not every prime factor of
$|\delta|$ can be congruent modulo $4l$ to an element of
$\mathcal{N}_1(l)$. Thus, $|\delta|$ must possess an odd prime divisor $p$ such that $p \equiv u \pmod{4l}$ for some $u \in \mathcal{N}_2(l)$. Therefore, $\left(\frac{l}{p}\right) = \left(\frac{l}{u}\right) = -1$, by Lemma~\ref{lemma:jacobi_periodicity}. Finally, since $p \mid |\delta|$ and $l \nmid |\delta|$, it follows that $\gcd(p, l) = 1$.

Since $p \mid l\delta =(x - 2cmn) $, evaluating~\eqref{eq:g3c-Codd} modulo $p$ yields
\[
\begin{aligned}
w^2 &\equiv w^2 - (x^2 - 4c^2m^2n^2)Z^2 \pmod{p} \\
&\equiv  acm^{2}\bigl(x^{d}-(2cmn)^{d}\bigr) + lc^{s+1}m^{2}t^{2q} \pmod{p} \\
&\equiv l c^{s+1} m^{2} t^{2q} \pmod{p} .
\end{aligned}
\]
By hypothesis, every prime factor $\mu$ of $cmt$ is such that $\mu \equiv r\pmod{4l}$, where $r\in \mathcal{N}_1(l)$, which means that $\left(\frac{l}{\mu}\right) =\left(\frac{l}{r}\right) = 1$. Since $\left(\frac{l}{p}\right) = -1$, we have $p \nmid cmt$. Defining $T \coloneqq c^{(s+1)/2} m t^q$, we see that $(wT^{-1})^{2} \equiv l \pmod{p}$. Since $p \nmid l$, this forces $\left(\frac{l}{p}\right) = 1$, a contradiction.

\medskip
\noindent\textbf{Subcase 2.2: $x \equiv -2cmn \pmod l$.}\quad

Since $d$ is odd, we have $x^d -(2cmn)^{d}\equiv (-2cmn)^d-(2cmn)^{d} \equiv -2(2cmn)^d \pmod l$. 

Reducing the right-hand side of~\eqref{eq:g3c-Codd} modulo $l$ yields
\[acm^{2}\bigl(x^{d}-(2cmn)^{d}\bigr) + lc^{s+1}m^{2}t^{2q}  \equiv -2acm^{2}(2cmn)^d \pmod l.\]
Since $a \equiv c \equiv mn \equiv 1 \pmod l$,~\eqref{eq:g3c-Codd} simplifies to $w^2 \equiv -2^{d+1} m^{2}\pmod l$, that is,
\[(wn)^2 \equiv -2^{d+1} (mn)^{2}\equiv -2^{d+1}\pmod l.\]

Since $d$ and $l$ are odd positive integers, we have $R \coloneqq 2^{\frac{d+1}{2}} \not\equiv 0 \pmod l$. Therefore, $(wn)^2 \equiv -R^2 \pmod l$, and hence $(wnR^{-1})^2 \equiv -1 \pmod l$. This forces $\left(\frac{-1}{l}\right) = 1$, a contradiction.

\medskip
\noindent\textbf{Subcase 2.3: $x \equiv 2cmn + v \pmod{4l}$} for some \textbf{$v \in \mathcal{N}_2(l)$.}\quad

Assume~\eqref{eq:main_eq} admits an integer solution $(x,y,z) $, where $M \coloneqq x - 2cmn \equiv v \pmod{4l}$ for some $v \in \mathcal{N}_2(l)$. We evaluate the Jacobi symbol for the positive odd integer $|M|$.

If $M > 0$, then $|M| \equiv v \pmod{4l}$, and Lemma~\ref{lemma:jacobi_periodicity} immediately yields $\left(\frac{l}{|M|}\right) = \left(\frac{l}{v}\right) = -1$.

If $M < 0$, then $|M| \equiv -v \pmod{4l}$. Because $l \equiv 3 \pmod 4$, the Legendre symbol evaluates to $\left(\frac{|M|}{l}\right) = \left(\frac{-v}{l}\right) = -\left(\frac{v}{l}\right)$. Note that $|M|$ and $v$ are both odd, and the condition $|M| \equiv -v \pmod 4$ implies that one of $|M|$ and $v$ is congruent to $1\pmod{4}$, whereas the other is congruent to $3\pmod{4}$. Consequently, $(-1)^{\frac{|M|-1}{2}} = -(-1)^{\frac{v-1}{2}}$. Applying the generalised law of quadratic reciprocity together with the fact that $v \in \mathcal{N}_2(l)$, we have
\[
  \left(\frac{l}{|M|}\right) \;=\; \left(-(-1)^{\frac{v-1}{2}}\right) \left(-\left(\frac{v}{l}\right)\right) \;=\; (-1)^{\frac{v-1}{2}}\left(\frac{v}{l}\right) \;=\; \left(\frac{l}{v}\right) \;=\; -1.
\]

Therefore, in either case, $\left(\frac{l}{|M|}\right) = -1$. By the multiplicativity of the Jacobi symbol, there exists an odd prime factor $p$ of $ |M|$ (i.e., $p \mid M$) such that $\left(\frac{l}{p}\right) = -1$. However, proceeding identically to Subcase~2.1, reduction of~\eqref{eq:g3c-Codd} modulo $p$ yields $(wT^{-1})^2 \equiv l \pmod p$, forcing $\left(\frac{l}{p}\right) = 1$, a contradiction.
\end{proof}
\begin{remark}
By symmetry, Theorem~\ref{thm:main_general_l} also holds if $cmt$ is replaced by $cnt$.
\end{remark}
\begin{corollary}\label{thm:main_general_l_k}
Let $k$ be an odd positive integer. Let the parameters $l, a, c, d, m, n, q, s$, and $t$ satisfy the hypotheses of Theorem~\ref{thm:main_general_l}. Set $b = a (2 c m n)^{d} - l^{k} c^{s} t^{2q}.$ Then~\eqref{eq:main_eq} has no integer solutions $(x,y,z)$ for which $x$ satisfies any of the following conditions:
\begin{enumerate}[label=(\alph*)]
    \item $x$ is even; \label{cond:1_k}
    \item $x \equiv 3 \pmod 4$; \label{cond:2_k}
    \item $x \equiv -2cmn \pmod l$; \label{cond:3_k}
    \item $x - 2cmn \equiv v \pmod{4l}$ for some $v \in \mathcal{N}_2(l)$. \label{cond:4_k}
\end{enumerate}
\end{corollary}
\begin{proof}
Since $l$ and $k$ are odd positive integers, $l^k \equiv l \pmod{4}$ and $\left(\frac{l^k}{p}\right) = \left(\frac{l}{p}\right)$ for any prime $p \nmid l$. Therefore, the proof follows verbatim from that of Theorem~\ref{thm:main_general_l} except for Subcase~2.1.
\end{proof}
\begin{remark}
    In Subcase~2.1 of the proof of Theorem~\ref{thm:main_general_l}, reducing~\eqref{eq:step2} modulo $l$ yields the term $l^{k-1}c^s t^{2q}$. Since this term vanishes modulo $l$ for $k \ge 2$, it imposes no constraint. Consequently, solutions satisfying $x \equiv 2cmn \pmod{l}$ cannot be excluded in Corollary~\ref{thm:main_general_l_k}.
\end{remark}
\begin{remark}
    Note that \cite[Corollary~2.2]{PrakashChakraborty} follows directly from~\ref{cond:3_k} and~\ref{cond:4_k} of Corollary~\ref{thm:main_general_l_k} by setting $l=3$ and $c=m=n=s=t=q=1$. 
\end{remark}
The following remark shows that the hypothesis $l \equiv 3 \pmod 4$ in Theorem~\ref{thm:main_general_l} cannot be extended to primes congruent to $1\pmod{4}$. 
\begin{remark}\label{imposibillity l5}
Setting $(a,c,m,n,t,d,q,s,l) = (1,1,1,1,1,5,1,1,5)$ in Theorem~\ref{thm:main_general_l}, we have $b = a (2cmn)^{d} -
l c^{s} t^{2q} = 2^{5} - 5 = 27$, and~\eqref{eq:main_eq} becomes
$  x^5 - y^2 - z^2 + xyz - 27 = 0$, which has the integer solution $(x,y,z) = (4,107,26)$, violating~\ref{cond:1} of Theorem~\ref{thm:main_general_l}.
\end{remark}

\section{The case \texorpdfstring{$l = -1$}{l = -1} and small primes \texorpdfstring{$l \in \{3,7, 11, 19\}$}{l = 3,7, 11, 19}}\label{sec:smallprimes}
In this section, we investigate the insolvability of~\eqref{eq:main_eq} corresponding to $b = a (2 c m n)^{d} - l c^{s} t^{2q}$ with $l \in \{-1, 3, 7, 11, 19\}$. For $l = -1$, Theorem~\ref{thm:general-plus1} restricts any possible solution of~\eqref{eq:main_eq} to the congruence class $x \equiv 1 \pmod 4$ and to the region $x \le 2cmn$. In contrast, Theorem~\ref{thm:main3_gen} establishes that~\eqref{eq:main_eq} admits no integer solutions for $l = 3$. Finally, for each prime $l\in\{7,11,19\}$, the explicit expressions for $\mathcal{N}_1(l)$ yield insolvability theorems (see Theorems~\ref{thm:main7_gen},~\ref{thm:main11_gen}, and~\ref{thm:main19_gen}) that restrict any possible integer solution $(x,y,z)$ of~\eqref{eq:main_eq} to a few exceptional residue classes of  $x$  modulo $4l$.

\subsection{The case \texorpdfstring{$l = -1$}{l = -1}.}
\label{sec:general+1c}
In this case, Theorem~\ref{thm:general-plus1} establishes that any integer solution to~\eqref{eq:main_eq} must satisfy specific congruence conditions and certain upper bounds on $x$.
\begin{theorem}\label{thm:general-plus1}
Let $a, c, m, n, t \in \ZZ$ and $d, q, s \in \ZZ_{>0}$ be parameters such that $d, m, n$, and $s$ are odd with $d\ge 3$. Assume that $a \equiv c \equiv 1 \pmod{4}$, and that every prime factor of the product $cmt$ is congruent to $1 \pmod{4}$. Set $b = a  (2c m n)^d + c^s t^{2q}$. Then~\eqref{eq:main_eq} has no integer solutions $(x,y,z)$ for which $x$ satisfies any of the following conditions:
\begin{enumerate}[label=(\alph*)]
    \item $x$ is even;
    \item $x \equiv 3 \pmod 4$;
    \item $x \equiv 1 \pmod 4$ and $x > 2cmn$. \label{condtion x>2cmn}
\end{enumerate}
\end{theorem}
\begin{proof}
Proceeding identically to the proof of Theorem~\ref{thm:main_general_l},~\eqref{eq:main_eq} admits no integer solutions when $x$ is even or $x \equiv 3 \pmod 4$, forcing $x \equiv 1 \pmod 4$. 

Since $c \equiv 1 \pmod 4$ and $mn$ is odd, we have $2cmn \equiv 2 \pmod 4$. Assume that ~\eqref{eq:main_eq} has an integer solution $(x,y,z)$ with $x>2cmn$. Setting $M \coloneqq x - 2cmn$, we have $M>0$ and $M  \equiv 1 - 2 \equiv 3 \pmod 4$. Consequently, $\left(\frac{-1}{M}\right) = (-1)^{\frac{M-1}{2}} = -1$ (see \cite[p.~192]{Burton}). By the multiplicativity of the Jacobi symbol, there exists an odd prime factor $p$ of $ M$ (i.e., $p \mid M$) such that $\left(\frac{-1}{p}\right) = -1$, which implies $p \equiv 3 \pmod 4$. 

Reducing~\eqref{eq:g3c-Codd} modulo $p$ yields $w^2 \equiv -T^2 \pmod p$, where $T \coloneqq c^{(s+1)/2} m t^q \not\equiv 0 \pmod p$ by hypothesis. This implies that $(wT^{-1})^2 \equiv -1 \pmod p$, which is impossible since $-1$ is a quadratic nonresidue for $p \equiv 3 \pmod 4$, providing a contradiction.
\end{proof}
\begin{remark}
To demonstrate that Theorem~\ref{thm:general-plus1} is sharp, we set $(a,c,m,n,t,d,q,s)=(-3,5,1,1,5,3,1,1)$, which yields $b=a(2cmn)^{d}+c^{s}t^{2q}=(-3)(2\cdot5)^{3}+5\cdot 5^{2}=-2875$ and $2cmn=10$. These parameters satisfy all the hypotheses of Theorem~\ref{thm:general-plus1}. With this choice of parameters,~\eqref{eq:main_eq} reduces to $ -3x^{3} - 5\bigl(y^{2} + z^{2}\bigr) + xyz + 2875 = 0$, for which $(x,y,z)=(9,-8,4)$ is a solution. Here, $x\equiv 1\pmod{4}$ and $x=2cmn-1$, which shows that the bound $x>2cmn$ in Theorem~\ref{thm:general-plus1}\ref{condtion x>2cmn} cannot be weakened.
\end{remark}
\subsection{The case \texorpdfstring{$l = 3$}{l = 3}}
\label{sec:general-3c}
As remarked earlier, the case corresponding to $l = 3$ is unique among other small primes examined in this section, as Theorem~\ref{thm:main3_gen} guarantees complete insolvability for~\eqref{eq:main_eq}.

\begin{theorem}\label{thm:main3_gen}
Let $a, c, m, n, t \in \ZZ$ and $d, q, s \in \ZZ_{>0}$ be such that $d, m, n$, and $s$ are odd, and $3 \mid d$. Assume that $a \equiv c \equiv 1 \pmod{12}$ and $mn \equiv 1 \pmod{3}$, and suppose that every prime factor of $cmt$ is congruent to $\pm 1 \pmod{12}$. Set $b = a  (2c m n)^d - 3c^s t^{2q}$. Then~\eqref{eq:main_eq} has no integer solutions $(x,y,z) $.
\end{theorem}
\begin{proof}
     For $l=3$, by Proposition~\ref{prop:leg3}, we have $\mathcal{N}_1(l) = \{1, 11\} $ and $\mathcal{N}_2(l) = \{5, 7\} $. By applying~\ref{cond:1},~\ref{cond:2}, and~\ref{cond:3} of Theorem~\ref{thm:main_general_l} with $l=3$, any integer solution $(x,y,z)$ of~\eqref{eq:main_eq} must simultaneously satisfy $x\equiv1\pmod{4}$ and $x \equiv 0\pmod{3}$. Therefore, by CRT, $x \equiv 9\pmod{12}$, which shows that $x-2cmn\equiv9-2\equiv7\pmod{12}$, where $7\in \mathcal{N}_{2}(l)$, which is impossible by Thoerem~\ref{thm:main_general_l}\ref{cond:4}. Thus, no such integer solution exists.
\end{proof}

\begin{corollary}\label{cor:3cs}
Let $a, c,d,s \in \ZZ$ be such that $d$ and $s$ are odd positive integers and $3 \mid d$. Assume that $a \equiv c \equiv 1 \pmod{12}$ and every prime factor of $c$ is congruent to $\pm 1 \pmod{12}$. Set $b = 2^d a c^d - 3c^s$. Then the Diophantine equation
\begin{equation}
    ax^d - c(y^2 + z^2) + xyz - b = 0
\end{equation}
has no integer solutions $(x,y,z) $. In particular, no such solutions exist when $s = 1$, which corresponds to $b = 2^d a c^d - 3c$.
\end{corollary}

\begin{proof}
This follows from Theorem~\ref{thm:main3_gen} with $m=n=t=q=1$. 
\end{proof}
\begin{remark}
    Note that \cite[Theorem~2.1]{PrakashChakraborty}, as well as \cite[Theorem~1.3]{VaishyaSharma} and \cite[Theorem~1]{MajumdarSury}, follow immediately from Corollary~\ref{cor:3cs} by setting $c=1$ alongside specific choices for the other parameters. Furthermore, as shown in \cite[Example~2]{PrakashChakraborty}, the constant $3$ in the definition of $b$ in Theorem~\ref{thm:main3_gen} cannot be replaced by an arbitrary power of $3$, because doing so would allow the~\eqref{eq:main_eq} to admit possible integer solutions.
\end{remark}

\subsection{The case \texorpdfstring{$l = 7$}{l = 7}}
The explicit expression for $\mathcal{N}_1(7)$ yields the following theorem.
\begin{theorem}\label{thm:main7_gen}
Let $a, c, m, n, t \in \ZZ$ and $d, q, s \in \ZZ_{>0}$ be such that $d, m, n$, and $s$ are odd, and $7 \mid d$. Assume that $a \equiv c \equiv 1 \pmod{28}$ and $mn \equiv 1 \pmod{7}$, and suppose that every prime factor of $cmt$ is congruent to $r \pmod{28}$, where $r\in  \{1, 3, 9, 19, 25, 27\}$. Set $b = a  (2c m n)^d - 7c^s t^{2q}$. Then~\eqref{eq:main_eq} has no integer solutions $(x,y,z) $ with $x \not\equiv 1,21 \pmod{28}$.
\end{theorem}
\begin{proof}
    For $l=7$, by Proposition~\ref{prop:leg7}, we have $\mathcal{N}_1(l) = \{1, 3, 9, 19, 25, 27\} $ and $\mathcal{N}_2(l) = \{ 5, 11, 13, 15, 17, 23\}  $. By applying~\ref{cond:1},~\ref{cond:2} and~\ref{cond:3} of Theorem~\ref{thm:main_general_l} with $l=7$, any integer solution $(x,y,z)$ of~\eqref{eq:main_eq} must satisfy $x\equiv1\pmod{4}$ and $x \equiv 0,1,3,4,6\pmod{7}$. Therefore, by CRT, $x \equiv 1,13,17,21,25\pmod{28}$, which shows that $x-2cmn\equiv11,15,19,23,27\pmod{28}$. Note that the cases $x - 2cmn \equiv 11, 15, 23 \pmod{28}$ are ruled out by Theorem~\ref{thm:main_general_l}\ref{cond:4}. Therefore, $x-2cmn\equiv19,27\pmod{28}$, that is, $x\equiv 1,21\pmod{28}$ are the only two possibilities for an integer solution $(x,y,z)$ to~\eqref{eq:main_eq}.
\end{proof}
\begin{example}\label{ex:residue-21}
To see that~\eqref{eq:main_eq} can indeed possess valid integer solutions for the congruence classes $x \equiv 1 \pmod{28}$ and $x \equiv 21 \pmod{28}$, we set the parameters $a=c=m=n=q=s=t=1$ and $d=7$. This choice satisfies all hypotheses of Theorem~\ref{thm:main7_gen}. The constant $b$ simplifies to $b = 2^7 - 7 = 121$, reducing~\eqref{eq:main_eq} to
\begin{equation}\label{eq:ex_residue_21}
    x^7 - (y^2+z^2) + xyz = 121.
\end{equation}
Explicit integer solutions $(x, y, z)$ to~\eqref{eq:ex_residue_21} for the following residue classes are as follows:
\begin{enumerate}[label=(\alph*)]
\item for \textbf{$x \equiv 1 \pmod{28}$}: $(-391, -1288032364, 6068448)$;
    \item for \textbf{$x \equiv 21 \pmod{28}$}: $(-7, 348, -484)$.
\end{enumerate}
This confirms that the restrictions $x \not\equiv 1, 21 \pmod{28}$ in Theorem~\ref{thm:main7_gen} are sharp.
\end{example}

\subsection{The case \texorpdfstring{$l = 11$}{l = 11}}
The explicit expression for $\mathcal{N}_1(11)$ yields the following theorem.
\begin{theorem}\label{thm:main11_gen}
Let $a, c, m, n, t \in \ZZ$ and $d, q, s \in \ZZ_{>0}$ be such that $d, m, n$, and $s$ are odd, and $11 \mid d$. Assume that $a \equiv c \equiv 1 \pmod{44}$ and $mn \equiv 1 \pmod{11}$, and suppose that every prime factor of $cmt$ is congruent to $r \pmod{44}$, where $r\in  \{1, 5,7, 9, 19, 25, 35,37,39,43\}$. Set $b = a  (2c m n)^d - 11c^s t^{2q}$. Then~\eqref{eq:main_eq} has no integer solutions $(x,y,z) $ with $x \not\equiv 1, 21, 37, 41 \pmod{44}$.
\end{theorem}
\begin{proof}
    For $l=11$, by Proposition~\ref{prop:leg11}, we have $\mathcal{N}_1(l) = \{1, 5,7, 9, 19, 25, 35,37,39,43\} $ and $\mathcal{N}_2(l) = \{ 3, 13, 15, 17, 21, 23,27,29,31,41\}$. By~\ref{cond:1},~\ref{cond:2} and~\ref{cond:3} of Theorem~\ref{thm:main_general_l} with $l=11$, any integer solution $(x,y,z)$ of~\eqref{eq:main_eq} must satisfy $x\equiv1\pmod{4}$ and $x \equiv 0,1,3,4,5,6,7,8,10\pmod{11}$. Therefore, by CRT, $x \equiv 1,5,17,21,25,29,33,37,41\pmod{44}$, which shows that $x-2cmn\equiv3,15,19,23,27,31,35,39,43\pmod{44}$. Note that the cases $x - 2cmn \equiv 3,15,23,27,31 \pmod{44}$ are ruled out by Theorem~\ref{thm:main_general_l}\ref{cond:4}. Therefore, $x-2cmn\equiv19,35,39,43\pmod{44}$, that is, $x\equiv 1,21,37,41\pmod{44}$ are the only possibilities for an integer solution $(x,y,z)$ to~\eqref{eq:main_eq}.
\end{proof}
\begin{example}
To see that~\eqref{eq:main_eq} can indeed possess valid integer solutions for the congruence classes $x \equiv 1, 21, 37, 41 \pmod{44}$, we consider the parameter family $a=c=m=n=q=s=1$ and $d=11$. This choice satisfies all hypotheses of Theorem~\ref{thm:main11_gen}. The constant $b$ simplifies to $b = a  (2c m n)^d - 11c^s t^{2q} = 2^{11} - 11t^2 = 2048 - 11t^2$, reducing~\eqref{eq:main_eq} to
\begin{equation}\label{eq:reduced_example}
    y^2 - xyz + z^2 = x^{11} - 2048 + 11t^2.
\end{equation}

Explicit integer solutions $(x, y, z, t)$ satisfying~\eqref{eq:reduced_example} for the following residue classes are as follows:

\begin{enumerate}[label=(\alph*)]
\item for \textbf{$x \equiv 1 \pmod{44}$:} $(1, 38, -10, 19)$ and $(1, 362, 488, 133)$; 
    \item for \textbf{$x \equiv 21 \pmod{44}$:} $(21, 3612744, 80078796, 445)$ and $(21, -12506232, 736128, 3829)$;
    \item for \textbf{$x \equiv 37 \pmod{44}$:} $(37, 8762096, -289694164, 5)$;
    \item for \textbf{$x \equiv 41 \pmod{44}$:} $(-3, -444, 972, 49)$.
\end{enumerate}
\end{example}

\subsection{The case \texorpdfstring{$l = 19$}{l = 19}}
The explicit expression for $\mathcal{N}_1(19)$ yields the following theorem.

\begin{theorem}\label{thm:main19_gen}
Let $a, c, m, n, t \in \ZZ$ and $d, q, s \in \ZZ_{>0}$ be such that $d, m, n$, and $s$ are odd, and $19 \mid d$. Assume that $a \equiv c \equiv 1 \pmod{76}$ and $mn \equiv 1 \pmod{19}$, and suppose that every prime factor of $cmt$ is congruent to $r \pmod{76}$, where $r\in  \{1, 3, 5, 9, 15, 17, 25, 27, 31, \allowbreak  45, 49, 51, 59, 61, 67, 71, 73, 75\}$. Set $b = a  (2c m n)^d - 19c^s t^{2q}$. Then~\eqref{eq:main_eq} has no integer solutions $(x,y,z) $ with $x \not\equiv 1,5,29,33,53,61,69,73 \pmod{76}$.
\end{theorem}
\begin{proof}
For $l=19$, Proposition~\ref{prop:leg19} yields
\begin{align*}
    \mathcal{N}_1(l) &= \{1, 3, 5, 9, 15, 17, 25, 27, 31, 45, 49, 51, 59, 61, 67, 71, 73, 75\}, \\
    \text{and} \quad \mathcal{N}_2(l) &= \{7, 11, 13, 21, 23, 29, 33, 35, 37, 39, 41, 43, 47, 53, 55, 63, 65, 69\}.
\end{align*}
    By~\ref{cond:1},~\ref{cond:2} and~\ref{cond:3} of Theorem~\ref{thm:main_general_l} with $l=19$, any integer solutions $(x,y,z) $ of~\eqref{eq:main_eq} must satisfy $x \equiv 1 \pmod{4}$ and $x \equiv 0, 1, 3, 4, 5, 6, 7, 8, 9, 10, 11, 12, 13, 14, 15, 16, 18 \pmod{19}$.
    
Therefore, by CRT, $x \equiv 1, 5, 9, 13, 25, 29, 33, 37, 41, 45, 49, 53, 57, 61, 65, 69, 73 \pmod{76}$, which shows that $x-2cmn \equiv 3, 7, 11, 23, 27, 31, 35, 39, 43, 47, 51, 55, 59, 63, 67, 71, 75 \pmod{76}$.

Note that the cases $x - 2cmn \equiv 7, 11, 23, 35, 39, 43, 47, 55, 63 \pmod{76}$ are ruled out by Theorem~\ref{thm:main_general_l}\ref{cond:4}. Therefore, $x-2cmn \equiv 3, 27, 31, 51, 59, 67, 71, 75 \pmod{76}$, that is, $ x \equiv 1, 5, 29, 33, 53, 61, 69, 73 \pmod{76}$ are the only possibilities for an integer solution $(x,y,z)$ to~\eqref{eq:main_eq}.
\end{proof}
\begin{example}
To see that~\eqref{eq:main_eq} can indeed possess valid integer solutions for the congruence classes $x \equiv 1, 5, 29, 33, 53, 61, 69, 73 \pmod{76}$, we consider the parameter family $a=c=m=n=q=s=1$ and $d=19$. This choice satisfies all hypotheses of Theorem~\ref{thm:main19_gen}. The constant $b$ simplifies to $b = 2^{19} - 19t^2 = 524288 - 19t^2$, reducing~\eqref{eq:main_eq} to
\begin{equation}\label{eq:reduced_example_76}
    y^2 - xyz + z^2 = x^{19} - 524288 + 19t^2.
\end{equation}
Explicit integer solutions $(x, y, z, t)$ satisfying~\eqref{eq:reduced_example_76} for the following residue classes are as follows:
\begin{enumerate}[label=(\alph*)]
    \item for \textbf{$x \equiv 1 \pmod{76}$}: 
    $(1, 318, 66, 179)$;
    
    \item for \textbf{$x \equiv 5 \pmod{76}$}: 
    $(5, 646300, -2995804, 79)$;
    
    \item for \textbf{$x \equiv 29 \pmod{76}$}: 
    $(29, \allowbreak 2745698438780, \allowbreak -47827361707540, 51)$;
    
    \item for \textbf{$x \equiv 33 \pmod{76}$}: 
    $(33, \allowbreak 35093598933292, \allowbreak -57464604408470, 5)$;
    
    \item for \textbf{$x \equiv 53 \pmod{76}$}: 
    $(53, \allowbreak -13727306672886548, \allowbreak 533798205413628, 1367)$ and 
    $(-23, \allowbreak 846143871766, \allowbreak -5331312857332, 79)$;
    
    \item for \textbf{$x \equiv 61 \pmod{76}$}: 
    $(61, \allowbreak 160538040203076864, \allowbreak 1780252738579908, 835)$ and 
    $(-15, \allowbreak 25334205200, \allowbreak -305298710066, 1)$;
    
    \item for \textbf{$x \equiv 69 \pmod{76}$}: 
    $(69, \allowbreak 617019272657740868, \allowbreak 6906457402209108, 4863)$ and 
    $(-7,\allowbreak -34464922, 164946544, 5)$;
    
    \item for \textbf{$x \equiv 73 \pmod{76}$}: 
    $(73, \allowbreak 11249789791824095722, \allowbreak 153827440733560778, 227)$ and 
    $(-3, 34212, -68648, 45)$.
\end{enumerate}
\end{example}
\section{Solutions coprime to \texorpdfstring{$cmn$}{cmn}}\label{sec:coprime}

We give an alternative formulation of Theorem~\ref{thm:main3_gen} in which the hypothesis that every prime factor of $cm$ is congruent to $r$ modulo $12$ for some $r \in
\mathcal{N}_1(3)$ is dropped. Removing this hypothesis restricts the complete insolvability result. Consequently, instead of excluding all integer solutions, Theorem~\ref{thm:coprime_3} only rules out those satisfying the specified congruence and coprimality restrictions.

\begin{theorem}\label{thm:coprime_3}
Let $a, c \in \ZZ$ such that $a \equiv c \equiv 1 \pmod{12}$. Let $d,m, n $, and $s$ be odd integers such that $mn \equiv 1 \pmod{3}$, and $d \ge 3$ be such that $3\mid d$. Let $s, q \in \ZZ_{>0}$, and suppose that $t \in \ZZ$ is such that every prime factor $p$ of $t$ is congruent to $\pm 1$ modulo $12$. Set $b = a (2cmn)^d - 3 c^s t^{2q}$. Then~\eqref{eq:main_eq} has no integer solutions $(x,y,z)$ for which $x$ satisfies any of the following conditions:
\begin{enumerate}[label=(\alph*)]
    \item $x$ is even;
    \item $x \equiv 3 \pmod 4$;
    \item $x \equiv 1 \pmod 4$ and $\gcd(x, cmn) = 1$.
\end{enumerate}
\end{theorem}
\begin{proof}
Proceeding as in the proof of Theorem~\ref{thm:main_general_l} with $l=3$, we have $b\equiv1\pmod{4}$, and any integer solution $(x,y,z)$ of~\eqref{eq:main_eq} must satisfy $x\equiv1\pmod{4}$ and $y \equiv z\equiv 0\pmod{2}$. Setting $y=2Y$, $z=2Z$, and $w = 2cm^{2}Y - x Z$,~\eqref{eq:main_eq} yields
\begin{equation}\label{eq:unified-Codd}
  w^2 - (x^2 - 4c^2m^2n^2)Z^2 \;=\; acm^2\bigl(x^d - (2cmn)^d\bigr) + 3 c^{s+1}m^2t^{2q}.
\end{equation}
Define $M \coloneqq x - 2cmn$. Furthermore, assume that $\gcd(x, cmn) = 1$, which ensures that $\gcd(M, cmn) =1$.

Since $m$ and $n$ are odd and $mn \equiv 1 \pmod 3$, it follows that $mn \equiv 1 \pmod 6$. Combined with $c \equiv 1 \pmod{12}$, this yields $cmn \equiv 1 \pmod 6$, and consequently, $2cmn \equiv 2 \pmod{12}$. Also, since $3\nmid m$, we have $m^2\equiv1\pmod{3}$.

Since $x \equiv 1 \pmod 4$, $x$ is restricted to the following residue classes modulo $12$. 

\begin{enumerate}[label=(\alph*)]
  \item \textit{$x \equiv 1 \pmod{12}$}: Since $cmn \equiv m^2\equiv1\pmod{3}$ and $d$ is odd, reducing~\eqref{eq:unified-Codd} modulo $3$ yields $w^2 \equiv 1-2^d\equiv 1-2\equiv 2 \pmod 3$, which is impossible.
  \item \textit{$x \equiv 5 \pmod{12}$}\label{case:b}: In this case, $M \equiv 3 \pmod{12}$. Consequently, $M=3\delta$, where $\delta\equiv1\pmod{4}$. Reduction of~\eqref{eq:unified-Codd} modulo $3$ yields $3\mid w$. Following Subcase $2.1$ in the proof of Theorem~\ref{thm:main_general_l}, we obtain a prime $p$ such that $p \mid \delta$ (i.e., $p \mid M$) and $\left(\frac{3}{p}\right) = -1$. Since $p \mid M$, it follows that $p \nmid cmn$ (by the gcd condition) and $p \nmid t$ (by Lemma~\ref{lem:t_odd}). Therefore, reduction of~\eqref{eq:unified-Codd} modulo $p$ yields $(w c^{-(s+1)/2} m^{-1} t^{-q})^2  \equiv 3 \pmod p$ (similar to the  Subcase~$2.1$ in the proof of Theorem~\ref{thm:main_general_l}), forcing $\left(\frac{3}{p}\right) = 1$, a contradiction.
  \item \textit{$x \equiv 9 \pmod{12}$}: In this case, $M \equiv 7 \pmod{12}$. Thus, $M$ has a prime factor $p$ such that $p \equiv \pm 5 \pmod{12}$. This implies that $\left(\frac{3}{p}\right) = -1$, which yields a contradiction using the same reasoning as in case~\ref{case:b}.
\end{enumerate}
\end{proof}
\begin{remark}\label{rem:sharpness}
To demonstrate Theorem~\ref{thm:coprime_3} is sharp, we consider a parameter set that satisfies all hypotheses but admits an integer solution $(x,y,z)$ with $x \equiv 1 \pmod 4$ and $\gcd(x, cmn) > 1$. Set $a=m=n=s=q=1$, $c=25$, $d=3$, and $t=11$. Therefore, $b = 2^3 \cdot25^3 - 3\cdot25 \cdot 11^2 = 125000 - 9075 = 115925$. These parameters satisfy all hypotheses of Theorem~\ref{thm:coprime_3}.

Substituting these values into~\eqref{eq:main_eq} yields $x^3 - 25\bigl(y^2 + z^2\bigr) + xyz - 115925 = 0$, which admits the integer solution $(x,y,z)=(65, 506, 230)$ satisfying $x \equiv 1 \pmod 4$ and $\gcd(x, 25) = 5 > 1$.
\end{remark}
\section{A generalised variant of \texorpdfstring{\eqref{eq:main_eq}}{}}\label{sec:asymmetric}
Generalising~\eqref{eq:main_eq}, we now consider the equation
\begin{equation}\label{eq:gen_eq}
  ax^{d} - py^{2} - qz^{2} + xyz - b = 0,
\end{equation}
where the coefficients of $y^2$ and $z^2$, namely $p$ and $q$, are no longer required to be related through a common factor $c$, but are allowed to vary independently subject to congruence conditions modulo $4$. Theorem~\ref{thm:main_general} restricts any possible solution $(x,y,z)$ of~\eqref{eq:gen_eq} to certain congruence classes modulo $4$ for $x$, depending on whether $a\equiv 1$ or $3\pmod{4}$ and on the parity of $d$. In particular, setting $p = cm^2$ and $q = cn^2$ recovers the insolvability result for the symmetric variant~\eqref{eq:main_eq} (see Corollary~\ref{thm:no_sol_general2}). A further specialisation yields analogous results for~\eqref{eq:gen_eq} (set to $p=q=1$) with $b = 2^d a - 3^k$ (see Corollary~\ref{cor:specific_a_3_1}).

\begin{theorem}\label{thm:main_general}
Let $l$ be a prime such that $l \equiv 3 \pmod 4$. Let $a,p, q, u, v \in \ZZ$ be such that $p \equiv q \equiv v \equiv 1 \pmod 4$. Let $r \in \ZZ_{\ge 0}$ and $d \ge 2$ be integers, and let $k$ and $t$ be odd integers with $k$ positive. Set $b = 2^{d}au - l^{k}vt^{2r}$. Then~\eqref{eq:gen_eq} has no integer solutions with $x$ even. Furthermore, the following statements hold:
\begin{enumerate}[label=(\alph*)]
\item If $a \equiv 3 \pmod 4$ and $d$ is even,~\eqref{eq:gen_eq} has no integer solutions.\label{Conition:1}
    \item If $a \equiv 1 \pmod 4$ and $d$ is odd,~\eqref{eq:gen_eq} has no integer solutions with $x \equiv 3 \pmod 4$. \label{Conition:2}
    \item If $a \equiv 3 \pmod 4$ and $d$ is odd,~\eqref{eq:gen_eq} has no integer solutions with $x \equiv 1 \pmod 4$.\label{Conition:3}
\end{enumerate}
\end{theorem}

\begin{proof}
We begin by evaluating the residue of $b$ modulo $4$. Because $k$ and $l$ are odd, we have $l^k \equiv l \pmod 4$. Since $t$ is odd, $t^{2r} \equiv 1 \pmod 4$. As $d \ge 2$ and $v \equiv 1 \pmod 4$, we have $b \equiv -l \equiv 1 \pmod 4$.

Suppose, for a contradiction, that there exists a solution $(x,y,z)$ with $x$ even. Since $d \ge 2$, we have $x^d \equiv 0 \pmod 4$. As $p, q$, and $b$ are odd, reducing~\eqref{eq:gen_eq} modulo $2$ yields
\[
 y^2 + z^2 \equiv 1 \pmod 2.
\]
Hence, $y$ and $z$ have opposite parity. Since $p \equiv q \equiv 1 \pmod 4$, we may assume without loss of generality that $y$ is even and $z$ is odd. Consequently, $y^2 \equiv 0 \pmod 4$, $z^2 \equiv 1 \pmod 4$, and $xyz \equiv 0 \pmod 4$. Since $q \equiv 1 \pmod 4$ and $x^d \equiv 0 \pmod 4$, reducing~\eqref{eq:gen_eq} modulo $4$ yields
\[
 0 \equiv -q - b \equiv -2 \pmod 4,
\]
a contradiction. Therefore, $x$ cannot be even.

\medskip
\noindent\textbf{Case~1}: \textit{$x$ is odd, with $a \equiv 3 \pmod 4$ and $d$ even.} Suppose, for a contradiction, that there exists a solution with $x$ odd. Since $a \equiv 3 \pmod 4$ and $b \equiv 1 \pmod{4}$, reducing~\eqref{eq:gen_eq} modulo $2$ yields 
\[
  a - (y^2 + z^2) + yz - 1 \equiv 0 \pmod 2 \implies y^2 + z^2 - yz \equiv 0 \pmod 2.
\]
Note that $y^2 + z^2 - yz\equiv 0 \pmod 2$ if and only if both $y$ and $z$ are even. Also, as $x$ is odd and $d$ is even, we have $x^d \equiv 1 \pmod 4$. Reducing~\eqref{eq:gen_eq} modulo $4$ yields 
\[
  ax^d  - b \equiv 0 \pmod 4 \implies a \equiv 1 \pmod 4,
\]
a contradiction. This proves~\ref{Conition:1}.

\medskip
\noindent\textbf{Case~2}: \textit{$x,a$  and $d$ are odd.}
Proceeding similarly as in Case~$1$, we deduce that both $y$ and $z$ are even. Also, as $d$ and $x$ are odd, we have $x^d \equiv x \pmod 4$. Reducing~\eqref{eq:gen_eq} modulo $4$ yields 
\[
 ax^d - b \equiv 0 \pmod 4 \implies ax \equiv 1 \pmod 4.
\]
Substituting either $a \equiv 1 \pmod 4$ with $x \equiv 3 \pmod 4$, or $a \equiv 3 \pmod 4$ with $x \equiv 1 \pmod 4$, yields $3 \equiv 1 \pmod 4$, which is a contradiction.  This proves~\ref{Conition:2} and~\ref{Conition:3}. 
\end{proof}
\begin{corollary}\label{thm:no_sol_general2}
Let $l$ be a prime with $l \equiv 3 \pmod 4$, and let $a,c \in \ZZ$ with $c \equiv 1 \pmod 4$. Let $r,s \in \ZZ_{\ge 0}$ and $d \ge 2$ be integers. Suppose $m, n, k$, and $t$ are odd integers with $k$ positive. Set $b = a(2cmn)^{d} - l^{k}c^{s}t^{2r}$. Then,~\eqref{eq:main_eq} has no integer solutions with $x$ even. Furthermore, the remaining conclusions of Theorem~\ref{thm:main_general} apply to~\eqref{eq:main_eq} as well.
\end{corollary}

\begin{proof}
Let $(p,q) \coloneqq (cm^2,cn^2)$. Then $p \equiv q \equiv 1 \pmod 4$, as $c \equiv m^2 \equiv n^2 \equiv 1 \pmod 4$.

Define $u = (c m n)^d$ and $v = c^s$. Therefore, we can write $b = a(2cmn)^{d} - l^{k}c^{s}t^{2r} = 2^{d}au - l^{k}vt^{2r}$.
The proof now follows immediately from Theorem~\ref{thm:main_general}. 
\end{proof}

\begin{corollary}\label{cor:specific_a_3_1}
 Let $a, d \in \ZZ$ with $d \ge 2$, and let $k$ be a positive odd integer. Set $b = 2^{d}a - 3^{k}$. Then, the Diophantine equation
\begin{equation}\label{eq:spec_a_3_1}
  ax^{d} - y^{2} - z^{2} + xyz - b= 0
\end{equation}
has no integer solutions with $x$ even. Furthermore, the remaining conclusions of Theorem~\ref{thm:main_general} apply to~\eqref{eq:spec_a_3_1} as well.
\end{corollary}
\begin{proof}
This proof is immediate from Corollary~\ref{thm:no_sol_general2} by taking $l = 3$ and $c = m = n = t = 1$. 
\end{proof}
\begin{remark}
Corollary~\ref{cor:specific_a_3_1} (with $k=1$) can be compared with Corollary~\ref{cor:3cs} (with $c=1$). In the latter, $d$ is an odd multiple of $3$, whereas $d$ can be taken to be even in Corollary~\ref{cor:specific_a_3_1}. 
\end{remark}

\section{Non-existence of positive integer solutions}\label{sec:positive solution}
In this section, rather than ruling out solutions within specific residue classes, we establish bounds on the size of any potential positive solution to~\eqref{eq:main_eq} for any positive integer $l$. In particular,~\eqref{eq:main_eq} admits only finitely many positive solutions satisfying $x \le 2cmn$.

\begin{theorem}
Let $a, c, d,l, m, n$ be positive integers such that $b \coloneqq a(2cmn)^d - l c>0$. Then~\eqref{eq:main_eq} has no positive integer solution $(x,y,z)$ such that $x\le \big(\frac{b}{a}\big)^{1/d}$.
\end{theorem}
\begin{proof}
Assume that~\eqref{eq:main_eq} has a positive integer solution $(x,y,z)$ such that $x\le \big(\frac{b}{a}\big)^{1/d}$, that is, $$c(m^2y^2 + n^2z^2) - xyz = ax^d -b.$$
Since $a, b > 0$, the quantity $R \coloneqq (b/a)^{1/d}$ is well-defined and positive. For $x \in (0, R]$, we have $x^d \le \frac{b}{a}$, that is, $ax^d - b\le 0$. Since $(my - nz)^2 \ge 0$ and $c> 0$, we obtain
$$2cmn yz \le c(m^2y^2 + n^2z^2) = xyz + ax^d -b \le xyz.$$
Since $yz > 0$, we have $2cmn \le x \le R$. Thus $2cmn \le R$, that is, $lc/a \leq 0$, a contradiction.
\end{proof}
\begin{theorem}\label{positive solutions bound}
Let $a, c, d, l, m, n$ be positive integers such that $b \coloneqq a(2cmn)^{d} - lc > 0$.
\begin{enumerate}[label=(\alph*)]
    \item If either $l = 1$ and $\gcd(m,n) > 1$, or $l \ge 2$ is squarefree, then~\eqref{eq:main_eq} has no integer solution $(x,y,z)$ with $x = 2cmn$.
    \item Every positive integer solution $(x,y,z)$ of~\eqref{eq:main_eq} with $x < 2cmn$ satisfies
    \[
    x > (b/a)^{1/d}, \qquad 2cmn - \frac{lc}{1+ad} \le x < 2cmn,
    \]
    \[
    yz \le \frac{lc}{2cmn - x} - ad, \qquad \text{and} \qquad |my - nz| \le \big\lfloor \sqrt{l} \big\rfloor.
    \]
    In particular, the number of such solutions is finite.
\end{enumerate}
\end{theorem}
\begin{proof}
Let $(x,y,z)$ be an integer solution where $x \le 2cmn$. Rearranging~\eqref{eq:main_eq} yields
\begin{equation}\label{eq:budget0}
(2cmn - x) yz + c(my-nz)^2 + a\bigl((2cmn)^d - x^d\bigr) = l c.
\end{equation}

\smallskip
\noindent{(a)} \,\textit{$x = 2cmn$.}\,
In this case,~\eqref{eq:budget0} simplifies to $c(my-nz)^2 = l c $, which implies $(my-nz)^2 = l$. Since $l\ge 2$ is squarefree, it cannot be a perfect square. Also, $l=1$ forces $\gcd(m,n)=1$, a contradiction. Thus, there are no integer solutions $(x,y,z)$ satisfying $x = 2cmn$.

\smallskip
\noindent{(b)} \,\textit{$x < 2cmn$.}\,
Let $(x,y,z)$ be a positive integer solution. Define $k = 2cmn - x$. By hypothesis, $k \ge 1$, so~\eqref{eq:budget0} becomes
\[
k yz + c(my-nz)^2 + a\bigl((2cmn)^d - x^d\bigr) = l c.
\]
Because $k ,y,z \ge 1$, we have $k yz + c(my-nz)^2 > 0$. Therefore,~\eqref{eq:budget0} yields $x > (b/a)^{1/d}$.

Since $a,d \ge 1$, the Lagrange's mean value theorem guarantees some $\gamma \in (x, 2cmn)$ such that
\[
a\bigl((2cmn)^d - x^d\bigr) = ad\gamma^{d-1}(2cmn-x) = kad\gamma^{d-1}.
\]
Because $\gamma > x \ge 1$, we have $\gamma^{d-1} \ge x^{d-1} \ge 1$, which yields $a\bigl((2cmn)^d - x^d\bigr) \ge kad$. Substituting this into~\eqref{eq:budget0}, we obtain
\[
l c = k yz + c(my-nz)^2 + a\bigl((2cmn)^d - x^d\bigr) \ge k yz + c(my-nz)^2 + kad.
\]
Since $a,c,d,k,y, z \ge 1$, every term on the right-hand side is nonnegative and bounded above by $l c$. In particular, we have $c(my-nz)^2 \le l c$, which yields $|my-nz| \le \big\lfloor \sqrt{l} \big\rfloor$. 

Similarly, we have $k(1+ad) \le k(yz+ad) \le l c$, which proves
$k \le \frac{l c}{1+ad}$. Since $x = 2cmn - k$, we have $2cmn - \frac{l c}{1+ad} \le x < 2cmn$.

Finally, $k(yz+ad) \le l c$ yields $yz \le \frac{l c}{k} - ad$, that is, $yz \le \frac{l c}{2cmn - x} - ad$.

For each of the finitely many integers $x$ in $2cmn - \frac{l c}{1+ad} \le x < 2cmn$, the condition $1\le yz \le \frac{l c}{2cmn - x} - ad$ restricts $y$ and $z$ to finitely many pairs. Consequently,~\eqref{eq:main_eq} has only finitely many positive integer solutions $(x,y,z)$ satisfying $x < 2cmn$.
\end{proof}
\begin{remark}
The bounds established in Theorem~\ref{positive solutions bound} are sharp. For instance, setting $l=2$, $a=c=d=1$, and $m=n=2$, ~\eqref{eq:main_eq} admits the solution $(x,y,z) = (7,1,1)$. This solution simultaneously attains the bounds for $x$ and $yz$, since $x = 2cmn - \frac{lc}{1+ad} = 7$ and $yz = \frac{lc}{2cmn-x} - ad = 1$. This optimal behaviour  is similarly exhibited by the parameter choice $l=3$, $a=2$, and $c=d=m=n=1$, where~\eqref{eq:main_eq} admits the solution $(x,y,z) = (1,1,1)$.

Furthermore, the difference bound is also realisable. An application of Theorem~\ref{positive solutions bound} with the choice of parameters $l=3$, $a=c=d=m=1$, and $n=2$, yields the solutions $(x,y,z) = (3,1,1)$ and $(3,2,1)$ for~\eqref{eq:main_eq}. For the specific solution $(3,1,1)$, the difference bound is attained, since $|my-nz| = |y-2z| = 1 = \lfloor\sqrt{3}\rfloor$. Under these parameters, we also observe that $x = \big\lceil 2cmn - \frac{lc}{1+ad} \big\rceil =3$ and the product bound holds as $yz \le \frac{lc}{2cmn-x} - ad = 2$.
\end{remark}

\section{Hyperelliptic curves}\label{sec:hyperelliptic}
In this section, we construct a family of (hyper)elliptic curves corresponding to~\eqref{eq:main_eq}, with  $a,c,d,m,n,q,s,t,b$ satisfying the hypotheses of Theorem~\ref{thm:main3_gen}, and study their rational torsion points using the insolvability of the associated Diophantine equation. The main tool is Grant's generalisation of the Nagell--Lutz theorem to hyperelliptic curves (Theorem~\ref{thm:grant}), which connects rational torsion points (defined later) to integer points on the curve. \subsection{Preliminaries on hyperelliptic curves and Grant's theorem}\label{subsec:hyper-preliminaries}
Grant's theorem \cite[Theorem~3]{Grant} is the genus-$g$ analogue of the classical Nagell--Lutz theorem \cite[Theorem~2.5, Remark~2.6]{SilvermanTate2015}, which was proved independently by T.\ Nagell \cite{Nagell1935} and \'E.\ Lutz \cite{Lutz1937}. To state the Nagell--Lutz theorem, we start with the definition of an elliptic curve (or, more generally, a hyperelliptic curve).

Throughout, $C$ denotes the smooth (i.e., nonsingular) projective curve over $\Q$ attached to an affine model $y^2=f(x)$, where $f\in\ZZ[x]$ is a monic, squarefree (see Definition \ref{def:discriminant}) polynomial of degree $\deg f\ge 3$, and its \textit{genus} is the integer $g=\lfloor \frac{\deg f-1}{2}\rfloor$. We call $C$ \textit{elliptic} over $\Q$ if $g=1$, and \textit{hyperelliptic} over $\Q$ if $g\ge 2$.

From now on, we assume that $\deg f$ is odd (i.e., $\deg f =2g+1$) to ensure that the smooth projective model $C$ contains exactly one point at infinity, denoted $\infty$. This establishes the setting for Grant's theorem below.

For a field $K \supseteq \Q$, the set of
\textit{$K$-rational points} on $C$, denoted $C(K)$, consists of the point at infinity together with all the solutions to $y^2=f(x)$ with coordinates in $K$, that is,
\begin{align}\label{C(K) defn}
   C(K)=\bigl\{(\alpha,\beta)\in K\times K : \beta^2=f(\alpha)\bigr\}\cup\{\infty\}.
\end{align}

An \textit{elliptic curve} $E$ over $\Q$ can be given by a Weierstrass equation of the form
\begin{align}\label{elliptic normal form}
    E\colon\; y^2=f(x)=x^3+Ax^2+Bx+C,\qquad A,B,C\in\ZZ,    
\end{align}
where $f$ is squarefree, so that $E$ is nonsingular. The rational points on $E$, denoted $E(\Q)$ (as defined in~\eqref{C(K) defn}), form a group under the chord–tangent group law, with $\infty$ (often written as $\mathcal{O}$) as its identity element.

For an elliptic curve $E$ (as in~\eqref{elliptic normal form}), the \textit{Nagell--Lutz theorem} says that every nontrivial rational torsion point $P=(x,y) \in E(\Q)$ has integer coordinates, that is, $x,y\in\ZZ$, and if $y\ne0$, then $y^2$ divides $\disc(f)$, where
$\disc(f)=18ABC-4A^3C+A^2B^2-4B^3-27C^2$. Since $\disc(f)\neq 0$ (by Remark~\ref{discriminant nonzero}), the condition $y^2 \mid \disc(f)$ ensures that there are only finitely many choices for $y$. Consequently, while the group $E(\Q)$ itself may be infinite, its torsion subgroup $E(\Q)_{\mathrm{tors}}$ is always finite. Mordell
\cite{Mordell1922} established that the group $E(\Q)$ is finitely generated. In sharp contrast, Faltings \cite{Faltings1983} proved the Mordell conjecture, establishing that for any smooth, projective, and absolutely irreducible curve $C$ over $\Q$ of genus $g \ge 2$, the set of rational points $C(\Q)$ is finite.

Unlike $E(\Q)$, the set $C(\Q)$ carries no natural group structure for a curve $C$ of genus $g \ge 2$. This obstruction is overcome by embedding $C$ into its Jacobian $J(C)$. To define the Jacobian, we first define divisors.

A (Weil) \textit{divisor} is a finite formal sum $\sum_i n_i P_i$, where $n_i \in \ZZ$ and $P_i \in C(\Qbar)$; equivalently, it is an element of the free abelian group on the set $C(\Qbar)$, and its \textit{degree} is $\sum_i n_i$.

To define principal divisors, we attach an integer $\ord_P(h)$, namely the \textit{order of $h$ at $P$}, to each point $P\in C(\Qbar)$ and each nonzero rational function $h$ on $C$. Since $C$ is smooth at $P$, there is a rational function $t$ vanishing to order exactly $1$ at $P$, called a \textit{local parameter} (or \textit{uniformizer}) at $P$. Any nonzero rational function $h$ can then be written uniquely as $h = t^{m} w$ with $m\in\ZZ$ and $w$ defined and nonzero at $P$, and we set $\ord_P(h)=m$. If $\ord_P(h)>0$, we say $h$ has a \textit{zero} at $P$ of that order; if $\ord_P(h)<0$, we say $h$ has a \textit{pole} at $P$ of order $-\ord_P(h)$.

A divisor is \textit{principal} if it is the divisor of a nonzero rational function
$h$ on $C$, namely $\ddiv(h)=\sum_P \ord_P(h)\,P$. Since a nonzero rational function
on $C$ has as many zeros as poles, counted with multiplicity, every principal divisor
has degree $0$; consequently the principal divisors form a subgroup
$\operatorname{Prin}(C)$ of the group $\operatorname{Div}^0(C)$ of degree-zero
divisors. Two divisors are \textit{linearly equivalent} if they differ by a principal
divisor. The \textit{Jacobian} of $C$ is the quotient group
\[
  J = J(C) := \operatorname{Div}^0(C)/\operatorname{Prin}(C),
\]
with addition induced from addition of divisors. For a degree-zero divisor
$D \in \operatorname{Div}^0(C)$, we write $[D] \in J$ for its linear-equivalence class;
thus $[D_1] = [D_2]$ if and only if $D_1 - D_2$ is principal, and the class $[0]$ of
the zero divisor is the identity element of $J$.

The absolute Galois group $G_{\Q}:=\operatorname{Gal}(\Qbar/\Q)$
acts on $C(\Qbar)$ coordinatewise, fixing $\infty$, and hence acts on divisors and on divisor classes. The group of \textit{rational points} $J(\Q)$ is the subgroup of degree-zero divisor classes (with $\Qbar$-coordinates) fixed by this action. Equivalently, $J(\Q)$ consists of divisor classes that admit a representative divisor whose points are permuted among themselves by every field automorphism of $\Qbar$ fixing $\Q$. This group $J(\Q)$ is called the Mordell--Weil group of $C$ (or of $J$) over $\Q$. By the Mordell--Weil theorem \cite{Weil1929}, $J(\Q)$ is a finitely generated abelian group, so its torsion subgroup $J(\Q)_{\mathrm{tors}}$ is finite. Equipped with this necessary terminology, we are now in a position to state Grant's theorem.

\begin{theorem}[{\cite[Theorem~3]{Grant}}] \label{thm:grant}
Let $C$ be the nonsingular projective curve of genus $g \ge 1$ given by $y^2 = f(x) = x^{2g+1} + b_1 x^{2g} + \dots + b_{2g} x + b_{2g+1}$, where $b_i \in \ZZ$ and $f$ is squarefree. Let
$\phi \colon C(\Q) \rightarrow J(\Q)$ be the Abel--Jacobi map, defined by $\phi(P) = [P - \infty]$. If $P = (x,y) \in C(\Q) \setminus \{\infty\}$ and
$\phi(P) \in J(\Q)_{\mathrm{tors}}$, then $x, y \in \ZZ$ and either $y = 0$ or $y^2$ divides $\disc(f)$. 
\end{theorem}

Let $C\colon y^2 = f(x)$ be a (hyper)elliptic curve over $\Q$ of genus $g \ge 1$, with $\deg f = 2g+1$, and let $\infty \in C(\Q)$ be the unique point at infinity. Let $\phi$ be the Abel--Jacobi map with basepoint $\infty$ as defined in Theorem~\ref{thm:grant}. We call a point $P \in C(\Q) \setminus \{\infty\}$ a \textit{rational torsion point} if $\phi(P) \in J(\Q)_{\mathrm{tors}}$. 

\subsection{From Diophantine equations to hyperelliptic curves}
For a fixed integer $k$, let $C_{k,a,b,c,m,n}$ be the affine plane
curve over $\Q$ defined by
\begin{align}\label{cccc}
  C_{k,a,b,c,m,n}:\qquad cm^2y^2 - kxy = ax^d - cn^2k^2 - b.
\end{align}
Since~\eqref{cccc} is precisely~\eqref{eq:main_eq} with $z=k$ fixed, a point $(x,y)\in\Q^2$ lies on $C_{k,a,b,c,m,n}$ if and only if $(x,y,k)$ is a solution of~\eqref{eq:main_eq}.

Assuming $d\ge3$, we now define the corresponding families of (hyper)elliptic curves based on the parity of $k$.

\noindent{(a)} \,\textit{$k$ is even.}\, 
 Multiplying~\eqref{cccc} by $cm^2$, and completing the square, we get    
 \begin{align*}
    \left(cm^2y - \frac{kx}{2}\right)^2 - \frac{k^2x^2}{4} &= cm^2ax^d - cm^2(cn^2k^2 + b).
\end{align*}
Multiplying by $(cm^2a)^{d-1}$, and using the fact that $d$ is an odd integer, we have
\begin{equation*}
    \left(\Big(cm^2y - \frac{kx}{2}\Big)(cm^2a)^{\frac{d-1}{2}}\right)^2 - \frac{(kx)^2(cm^2a)^{d-1}}{4} = (cm^2ax)^d - cm^2(cn^2k^2 + b)(cm^2a)^{d-1}.
\end{equation*}
Substituting $X = cm^2ax$ and $Y = \left(cm^2y - \frac{kx}{2}\right)(cm^2a)^{\frac{d-1}{2}}$, we obtain the following family of hyperelliptic (or elliptic when $d=3$) curves
\begin{align}\label{Hyperelliptic even}
    H_{k,a,b,c,m,n}^e : Y^2 - (cm^2a)^{d-3}\frac{k^2X^2}{4} = X^d - cm^2(cn^2k^2 + b)(cm^2a)^{d-1}.
\end{align}
    
\noindent{(b)} \,\textit{$k$ is odd.}\, 
Multiplying~\eqref{cccc} by $4cm^2$, and completing the square, we get
\begin{align*}
    (2cm^2y - kx)^2 - k^2x^2 &= 4cm^2ax^d - 4cm^2(cn^2k^2 + b).
\end{align*}
Multiplying this by $(4cm^2a)^{d-1}$ yields
\begin{align*}
    &\left((2cm^2y - kx)(4cm^2a)^{\frac{d-1}{2}}\right)^2 - (4cm^2a)^{d-1}k^2x^2 \\
    =& (4cm^2ax)^d - 4cm^2(cn^2k^2 + b)(4cm^2a)^{d-1}.
\end{align*}
 Finally, substituting $X = 4cm^2ax$ and $Y = (2cm^2y - kx)(4cm^2a)^{\frac{d-1}{2}}$, we obtain the following family of hyperelliptic (or elliptic when $d=3$) curves
\begin{align}\label{Hyperelliptic odd}
    H_{k,a,b,c,m,n}^o : Y^2 - (4cm^2a)^{d-3}k^2X^2 = X^d - 4cm^2(cn^2k^2 + b)(4cm^2a)^{d-1}.
\end{align}

Consider the family of (hyper)elliptic curves defined by
\begin{align}\label{Hyperelliptic}
H_{k,a,b,c,m,n} = 
\begin{cases} 
H_{k,a,b,c,m,n}^e & \text{if } k \text{ is even,} \\ 
H_{k,a,b,c,m,n}^o & \text{if } k \text{ is odd.} 
\end{cases}
\end{align}
We now analyse the rational torsion points of $H_{k,a,b,c,m,n}$.
\begin{lemma} \label{lem:inverse-mapping}
Let $a, c, d, m, n$ and $b$ satisfy the hypotheses of Theorem~\ref{thm:main3_gen}. A rational point $(X, Y) \in H_{k,a,b,c,m,n}(\Q)\setminus\{\infty\}$ bijectively corresponds to a solution $(x, y, k) \in \Q^3$ of~\eqref{eq:main_eq}. Furthermore, if $(X, Y) \in \ZZ^2$ is an integer point on the (hyper)elliptic curve, it yields an integer solution $(x, y, k)$ of~\eqref{eq:main_eq} if and only if the following divisibility conditions hold:
\begin{enumerate}[label=(\alph*)]
    \item for $k$ even, $\frac{X}{cm^2a}\in \ZZ$ and $\frac{Y}{(cm^2)^{\frac{d+1}{2}}a^{\frac{d-1}{2}}} + \frac{kX}{2c^2m^4a}\in \ZZ$;\label{Ccond:a}
    \item for $k$ odd, $\frac{X}{4cm^2a}\in \ZZ$ and $\frac{Y}{2^d(cm^2)^{\frac{d+1}{2}}a^{\frac{d-1}{2}}} + \frac{kX}{8c^2m^4a}\in \ZZ$.\label{Ccond:b}
\end{enumerate}
\end{lemma}
\begin{proof}
To establish the stated conditions, we examine the inverse transformations for both parities of $k$. Note that $H_{k,a,b,c,m,n}\cap \Q^2=H_{k,a,b,c,m,n}(\Q)\setminus\{\infty\}$.

\noindent{(a)} \,\textit{$k$ is even.}\,
Let $C_{k,a,b,c,m,n}$ and $H^{e}_{k,a,b,c,m,n}$ be defined as in~\eqref{cccc} and~\eqref{Hyperelliptic even} respectively. Define the maps  $ \phi^e \colon C_{k,a,b,c,m,n}\cap \Q^2 \rightarrow H^e_{k,a,b,c,m,n}\cap \Q^2$ and $  \psi^e \colon H^e_{k,a,b,c,m,n} \cap \Q^2\rightarrow C_{k,a,b,c,m,n}\cap \Q^2$ by
\begin{align}
  \phi^e(x,y)
    &= \left(\,cm^2ax,\ \left(cm^2y - \frac{kx}{2}\right)(cm^2a)^{\frac{d-1}{2}}\,\right),
      \label{even forward}
\end{align}
\begin{align}
  \psi^e(X,Y)
    &= \left(\,\frac{X}{cm^2a},\
        \frac{Y}{(cm^2)^{\frac{d+1}{2}}a^{\frac{d-1}{2}}}
        + \frac{kX}{2c^2m^4a}\,\right).
      \label{even invert}
\end{align}
Therefore, we have $\psi^e\circ\phi^e=\mathrm{id}$ and $\phi^e\circ\psi^e=\mathrm{id}$, giving the bijection. Moreover, writing $\psi^e(X,Y)=(x,y)$, we have $x=\frac{X}{cm^2a}$ and $y=\frac{Y}{(cm^2)^{(d+1)/2}a^{(d-1)/2}}+\frac{kX}{2c^2m^4a}$, so $(x,y)\in\ZZ^2$ if and only if~\ref{Ccond:a} holds.

\noindent{(b)} \,\textit{$k$ is odd.}\,
Let $H^{o}_{k,a,b,c,m,n}$ be defined as in~\eqref{Hyperelliptic odd}. Define the maps  $ \phi^{o} \colon C_{k,a,b,c,m,n}\cap \Q^2 \rightarrow H^{o}_{k,a,b,c,m,n}\cap \Q^2$ and $  \psi^{o} \colon H^{o}_{k,a,b,c,m,n} \cap \Q^2\rightarrow C_{k,a,b,c,m,n}\cap \Q^2$ by
\begin{align}
  \phi^{o}(x,y)
    &= \left(\,4cm^2ax,\ \left(2cm^2y -kx\right)(4cm^2a)^{\frac{d-1}{2}}\,\right),
      \label{odd forward}
\end{align}
\begin{align}
  \psi^{o}(X,Y)
    &= \left(   \frac{X}{4cm^2a},\frac{1}{2cm^2}\left(\frac{Y}{(4cm^2a)^{\frac{d-1}{2}}} +\frac{kX}{4cm^2a}\right)  \right).
      \label{odd invert}
\end{align} 
Therefore, we have $\psi^o\circ\phi^o=\mathrm{id}$ and $\phi^o\circ\psi^o=\mathrm{id}$, giving the bijection. Moreover, writing $\psi^o(X,Y)=(x,y)$, we have $x=\frac{X}{4cm^2a}$ and $y=\frac{Y}{2^d(cm^2)^{(d+1)/2}a^{(d-1)/2}}+\frac{kX}{8c^2m^4a}$, so $(x,y)\in\ZZ^2$ if and only if~\ref{Ccond:b} holds.
\end{proof}

\begin{lemma}\label{lem:squarefree-condition}
Let $a, c, m, n, b \in \ZZ$ and let $d \ge 3$ be an odd integer, as given in Theorem~\ref{thm:main3_gen}. Let $k \in \ZZ$ and let $f(X) = X^d + \lambda X^2 + \mu$, where
\[
(\lambda,\mu) =
\begin{cases}
    \bigl( \frac{k^2}{4}(cm^2a)^{d-3}, \, -cm^2(cn^2k^2+b)(cm^2a)^{d-1} \bigr), & \text{if } k \text{ is even}, \\[1ex]
    \bigl( k^2(4cm^2a)^{d-3}, \, -4cm^2(cn^2k^2+b)(4cm^2a)^{d-1} \bigr), & \text{if } k \text{ is odd}.
\end{cases}
\]
Then $f$ is squarefree if and only if $ cn^2k^2+b \neq 0$ and 
\begin{equation}\label{eq:squarefree-condition}
(d-2)^{d-2} k^{2d} \neq 4^{d-1} d^d (cm^2)^d a^2 (cn^2k^2+b)^{d-2}.
\end{equation}
\end{lemma}
\begin{proof}
By Lemma \ref{lem:trinomial_discriminant}, the discriminant is given by
\begin{equation*}
\disc(f)=(-1)^{(d-1)/2}\,\mu\,\bigl(4(d-2)^{d-2}\lambda^{d}+d^{d}\mu^{d-2}\bigr).
\end{equation*} 
Thus, $f$ is squarefree if and only if $\mu \neq 0$ and $4(d-2)^{d-2}\lambda^{d} \neq -d^{d}\mu^{d-2}$. 

Since $a,c,m\neq 0$ by hypothesis, $\mu \neq 0$ if and only if $ cn^2k^2+b \neq 0$ for both parities of $k$. We now establish the condition under which $4(d-2)^{d-2}\lambda^{d} \neq -d^{d}\mu^{d-2}$.

\medskip
\noindent\textbf{Case 1}: \textit{$k$ is even.}\quad
Substituting $\lambda$ and $\mu$ into $4(d-2)^{d-2}\lambda^{d} \neq -d^{d}\mu^{d-2}$ yields
\[ 4(d-2)^{d-2} \Bigl( \frac{k^2}{4}(cm^2a)^{d-3} \Bigr)^d \neq -d^d \Bigl( -cm^2(cn^2k^2+b)(cm^2a)^{d-1} \Bigr)^{d-2}. \]
Since $d$ is odd, the above is true if and only if
\[ (d-2)^{d-2} \frac{k^{2d}}{4^{d-1}} \neq d^d (cm^2)^{d-2} (cn^2k^2+b)^{d-2} (cm^2a)^2, \]
which holds if and only if~\eqref{eq:squarefree-condition} holds.

\medskip
\noindent\textbf{Case 2}: \textit{$k$ is odd.}\quad
Substituting $\lambda$ and $\mu$ into $4(d-2)^{d-2}\lambda^{d} \neq -d^{d}\mu^{d-2}$ yields
\[ 4(d-2)^{d-2} \bigl( k^2(4cm^2a)^{d-3} \bigr)^d \neq -d^d \bigl( -4cm^2(cn^2k^2+b)(4cm^2a)^{d-1} \bigr)^{d-2}. \]
Since $d$ is odd, the above is true if and only if
\[ 4(d-2)^{d-2} k^{2d} \neq d^d (4cm^2)^{d-2} (cn^2k^2+b)^{d-2} (4cm^2a)^2, \]
which holds if and only if~\eqref{eq:squarefree-condition} holds. 
\end{proof}
\begin{lemma}\label{lem:squarefree}
Let $a, c, m, n, b \in \ZZ$ and let $d \ge 3$ be an odd integer, as given in Theorem~\ref{thm:main3_gen}. For a fixed $k \in \ZZ$, consider the curve $H_{k,a,b,c,m,n}$ of~\eqref{Hyperelliptic}, defined by $Y^2 = f(X)$, where $f $ is the polynomial given as in the statement of Lemma~\ref{lem:squarefree-condition}. Then $f$ is squarefree.
\end{lemma}
\begin{proof}
In both parities for $k$ (odd or even), we have $f(X)=X^d+\lambda X^2+\mu$, where $(\lambda,\mu)$ is defined as in Lemma \ref{lem:squarefree-condition}. Observe that $\lambda \in \ZZ$ when $k$ is even. Consequently, $ f \in\ZZ[X]$ in both cases. For an odd integer $d\ge 3$, Lemma \ref{lem:trinomial_discriminant} yields
\begin{equation}
\disc(f)=(-1)^{(d-1)/2}\,\mu\,\bigl(4(d-2)^{d-2}\lambda^{d}+d^{d}\mu^{d-2}\bigr).
\end{equation}
By Theorem~\ref{thm:main3_gen}, $a$, $m$, $n$ are odd and $c\equiv1\pmod{4}$. In particular, $a$, $c$, $m$, and $n$ are nonzero. To show that $f$ is squarefree, we first prove that~\eqref{eq:squarefree-condition} holds for every choice of parameters satisfying the hypothesis. Assume for the sake of contradiction that
\begin{equation}\label{disc equality}
(d-2)^{d-2} k^{2d} = 4^{d-1} d^d (cm^2)^d a^2 (cn^2k^2+b)^{d-2}.
\end{equation}

Since $d \ge 3$, $c \equiv 1 \pmod 4$, and $t$ is odd (by Lemma~\ref{lem:t_odd}), it follows that $2^d \equiv 0 \pmod 4$ and $c^s \equiv t^{2q} \equiv 1 \pmod 4$, which yields $ b= a(2cmn)^d - 3c^s t^{2q} \equiv1\pmod{4}$. Therefore, $b$ is odd and $b\neq 0$.

If $k=0$, then~\eqref{disc equality} reduces to $0 = 4^{d-1} d^d (cm^2)^d a^2\, b^{d-2}$. Since $a, c, m,$ and $b$ are nonzero, and $d \ge 3$, we have $4^{d-1} d^d (cm^2)^d a^2\, b^{d-2}\neq 0$, which is a contradiction. Therefore, we assume $k\neq0$, so that the $2$-adic valuation of $k$, $v_2(k)$, is well-defined.

If $k$ is odd, then $(d-2)^{d-2}k^{2d}$ is odd. This means that $v_2((d-2)^{d-2}k^{2d})=0$, whereas the right-hand side of \eqref{disc equality} has $v_2(4^{d-1} d^d (cm^2)^d a^2 (cn^2k^2+b)^{d-2})\ge 2(d-1)\ge 4$, a contradiction.

If $k(\neq 0)$ is even, then $4\mid cn^2k^2$. We consider $2$-adic valuation for both sides of \eqref{disc equality}. Since $b$ is odd, $cn^2k^2+b$ is odd. Because $a$, $c$, and $m$ are also odd, and $v_2(k)\ge 1$, we have
\[
v_2(\mathrm{LHS})=2d\,v_2(k)\ge 2d >2(d-1)=v_2(\mathrm{RHS}),
\]
a contradiction. Therefore, in either case~\eqref{disc equality} does not hold, which ensures that~\eqref{eq:squarefree-condition} holds.

Since $n^2\equiv b\equiv c\equiv 1 \pmod{4}$ by hypotheses, we have $cn^2k^2 + b \equiv k^2+1 \not \equiv 0 \pmod 4$, yielding $cn^2k^2 + b \neq 0$. Therefore, by Lemma \ref{lem:squarefree-condition}, $f$ is squarefree.
\end{proof}
\begin{theorem} \label{thm:general-3c-hyperelliptic}
Let the parameters $a, c, d, m, n, s, q, t$, and $b$ satisfy the hypotheses of Theorem~\ref{thm:main3_gen}. For any fixed integer $k \in \ZZ$, consider the corresponding family of (hyper)elliptic curves $H_{k,a,b,c,m,n}$ over $\Q$ as defined in~\eqref{Hyperelliptic}. Then $H_{k,a,b,c,m,n}$ possesses no rational torsion point $(X,Y)$ satisfying
\begin{enumerate}[label=(\alph*)]
    \item $\frac{X}{cm^2a}\in \ZZ$ and $\frac{Y}{(cm^2)^{\frac{d+1}{2}}a^{\frac{d-1}{2}}} + \frac{kX}{2c^2m^4a}\in \ZZ$, when $k$ is even;\label{Cond: 1}
    \item $\frac{X}{4cm^2a}\in \ZZ$ and $\frac{Y}{2^d(cm^2)^{\frac{d+1}{2}}a^{\frac{d-1}{2}}} + \frac{kX}{8c^2m^4a}\in \ZZ$, when $k$ is odd.\label{Cond: 2}
\end{enumerate}
In particular, for an even integer $k$, we have the following.
\begin{enumerate}[label=(\roman*)]
    \item The family $H_{k,1,b,1,1,n}$ possesses no rational torsion points; that is, no $\widetilde{P}\in H_{k,1,b,1,1,n}(\Q)\setminus\{\infty\}$ satisfies $\phi(\widetilde{P})\in J(\Q)_{\mathrm{tors}}$.\label{HyperCond: 1}
    \item Consequently, $\phi\big(H_{k,1,b,1,1,n}(\Q)\big) \cap J(\Q)_{\mathrm{tors}}=\{[0]\}$.\label{HyperCond: 2}
    \item Furthermore, when $d=3$, $H_{k,1,b,1,1,n}$ has torsion-free Mordell--Weil group over $\Q$.\label{HyperCond: 3}
\end{enumerate}
\end{theorem}
\begin{proof}
The (hyper)elliptic curve $H_{k,a,b,c,m,n}$ is given by $Y^2 = f(X)$, where $f$ is the
polynomial defined in Lemma~\ref{lem:squarefree-condition}; by Lemma~\ref{lem:squarefree},
$f$ is squarefree.

Assume, for contradiction, that $H_{k,a,b,c,m,n}$ possesses a rational torsion point
$P=(X,Y)$ (i.e., $P \in H_{k,a,b,c,m,n}(\Q)\setminus\{\infty\}$ with
$\phi(P) \in J(\Q)_{\mathrm{tors}}$) satisfying the divisibility conditions of
either~\ref{Cond: 1} or~\ref{Cond: 2}. By Theorem~\ref{thm:grant}, $(X,Y)\in\ZZ^2$.
These conditions are precisely~\ref{Ccond:a} or~\ref{Ccond:b} of
Lemma~\ref{lem:inverse-mapping}, so the integer point $(X,Y)$ yields an integer solution
of~\eqref{eq:main_eq}, contradicting Theorem~\ref{thm:main3_gen}. Hence $H_{k,a,b,c,m,n}$
possesses no rational torsion point satisfying the respective conditions.

Suppose $H_{k,1,b,1,1,n}$ has a rational torsion
point $(X_{0},Y_{0})$. By Theorem~\ref{thm:grant}, $(X_{0},Y_{0}) \in \ZZ^{2}$. Since $k$ is even, $\frac{kX_{0}}{2}\in\ZZ$, so $\bigl(X_{0},\,Y_{0}+\frac{kX_{0}}{2}\bigr)\in\ZZ^2$. Therefore, $(X_0,Y_0)$ satisfies the divisibility conditions of~\ref{Cond: 1} with $a=c=m=1$. This contradicts case~\ref{Cond: 1} (with $a=c=m=1$), proved above, which ensures that $H_{k,1,b,1,1,n}$ has no rational torsion point $(X,Y)$ with $\bigl(X,\,Y+\frac{kX}{2}\bigr)\in\ZZ^2$. Hence $H_{k,1,b,1,1,n}$ has no rational torsion point when $k$ is even. This proves~\ref{HyperCond: 1}.

Let $[D] \in \phi\big(H_{k,1,b,1,1,n}(\Q)\big) \cap J(\Q)_{\mathrm{tors}}$, where $D \in \operatorname{Div}^0\!\big(H_{k,1,b,1,1,n}\big)$. Then $[D]=\phi(P)$ for some $P \in H_{k,1,b,1,1,n}(\Q)$. If $P\neq\infty$, then $P\in H_{k,1,b,1,1,n}(\Q)\setminus\{\infty\}$ and $\phi(P)=[D]\in J(\Q)_{\mathrm{tors}}$, contradicting~\ref{HyperCond: 1}. Therefore, $P=\infty$, whence $[D]=\phi(\infty)=[\infty-\infty]=[0]$. Consequently,  $\phi\big(H_{k,1,b,1,1,n}(\Q)\big) \cap J(\Q)_{\mathrm{tors}}=\{[0]\}$. This proves~\ref{HyperCond: 2}.

For $d=3$, the curve $ H_{k,1,b,1,1,n}$ is elliptic, so $H_{k,1,b,1,1,n}(\Q)$ is an abelian group (with identity $\infty$), and the Abel--Jacobi map $\phi\colon E(\Q)\to J(\Q)$, $P\mapsto[P-\infty]$, is a group isomorphism (see \cite[p.~330]{Silvermanarithmetic}). Hence, any nontrivial $P\in H_{k,1,b,1,1,n}(\Q)_{\mathrm{tors}}$ satisfies $\phi(P)\in J(\Q)_{\mathrm{tors}}$. This means that $P$ is a rational torsion point, contradicting~\ref{HyperCond: 1}. Therefore, $H_{k,1,b,1,1,n}(\Q)_{\mathrm{tors}}$ is trivial. This proves~\ref{HyperCond: 3}.
\end{proof}

\begin{remark}\label{rem:pc-correction}
Certain claims in \cite{VaishyaSharma} and \cite{PrakashChakraborty} are not fully justified by the arguments provided there. Throughout, our $k$ corresponds to $m$, and our $a$ to $A$, in \cite{VaishyaSharma}.

\textit{First}, in mapping an integer point $(X_0,Y_0)$ on the associated (hyper)elliptic curve back to an integer solution of the Diophantine equation, the authors used the substitution $x = aX_0$ for $k$ even (resp., $x = 4aX_0$ for $k$ odd); see the proof of \cite[Theorem~1.2]{VaishyaSharma} (where $(X_{0},Y_{0})$ is mapped to $(aX_{0},\,aY_{0}-\tfrac{akX_{0}}{2},\,k)$ for even $k$) and \cite[Theorem~3.3]{PrakashChakraborty} (where $(X_{0},Y_{0})$ is sent to $(aX_{0},\,(Y_{0}-\tfrac{kX_{0}}{2})a^{\frac{d-1}{2}},\,k)$ for even $k$). The correct inverse map is instead $x = \tfrac{X_0}{a}$ for $k$ even (resp., $x = \tfrac{X_0}{4a}$ for $k$ odd), which is an integer only when $a \mid X_0$ (resp., $4a \mid X_0$). For a torsion point of the curve to yield an \textit{integer} solution of the Diophantine equation, this divisibility must hold. However, it fails for general integer points on the curve. Consequently, the reduction to the Diophantine equation does not transform torsion points to integer solutions when $a \neq 1$, so the argument is incomplete for every odd $k$ and for even $k$ with $a \neq 1$. It works unconditionally when $a = 1$ and $k$ is even, since the inverse map is then integral; this is the case that we establish in Theorem~\ref{thm:general-3c-hyperelliptic}.

\textit{Second}, for the hyperelliptic curves of \cite{PrakashChakraborty} (those with $\deg f \ge 5$, i.e., genus $g \ge 2$), a further gap remains even if the inverse map is corrected, because the Abel--Jacobi map need not be surjective for $g \ge 2$. In particular, there is no guarantee that every class in $J(\Q)_{\mathrm{tors}}$ is of the form $[P - \infty]$ for some $P \in C(\Q)$, the only classes to which Grant's theorem applies. Consequently, the claim that such a hyperelliptic curve has torsion-free Mordell--Weil group cannot be obtained from Grant's theorem alone. Indeed, a nontrivial rational torsion class can be of the form $[P_1 + P_2 - 2\infty]$, where $\{P_1, P_2\}$ is a Galois-conjugate pair of non-rational points; such a class is not in general of the form $[P - \infty]$, so Grant's theorem does not apply to it. Hence, the argument for a torsion-free Mordell--Weil group does not extend to $g \ge 2$. By contrast, in the elliptic case ($\deg f = 3$), the Abel--Jacobi map is an isomorphism onto $J$. Therefore, every torsion class is represented by a rational torsion point and the argument goes through, although this further requires that the above divisibility constraint is satisfied. We carry out this analysis for $a = 1$ in Theorem~\ref{thm:general-3c-hyperelliptic}. In doing so, we demonstrate that the proof of \cite[Theorem~1.2 (1)]{VaishyaSharma}, as originally presented, requires the additional assumption that $a=1$, and the method presented there is insufficient to prove \cite[Theorem~1.2 (2)]{VaishyaSharma}. This discussion suggests that \cite[Corollary~1.4]{VaishyaSharma} also requires a valid proof.

\textit{Third}, to apply Grant's theorem to a (hyper)elliptic curve $Y^2 = f(X)$, we require $f$ to be squarefree, which we establish by showing that the discriminant of $f$ is nonzero (see Lemma~\ref{lem:squarefree}). We remark that the squarefreeness of $f$ was not explicitly mentioned in \cite{VaishyaSharma} (resp., \cite{PrakashChakraborty}) before applying the Nagell--Lutz (resp., Grant) theorem.

Although these gaps invalidate the proofs of \cite[Theorem~1.2]{VaishyaSharma} and \cite[Theorem~3.3]{PrakashChakraborty}, they do not necessarily make the underlying statements false. Indeed, it is possible that the statements may still be true and admit correct independent proofs. In particular, the computational evidence detailed in the following remark suggests that the statement of \cite[Theorem~1.2]{VaishyaSharma} may still hold.
\end{remark}

\begin{remark}\label{rem:our-scope-and-counterexample}
The divisibility obstruction of Remark~\ref{rem:pc-correction} also constrains our construction in the same way, and in fact bounds it sharply. For $d = 3$, $a = c = m = 1$, substituting $b = (2n)^3 - 3t^{2q}$ gives the explicit elliptic family
\[
  H_{k,1,b,1,1,n}\colon\ Y^2 = X^3 + \tfrac{k^2}{4}\,X^2 - \bigl(n^2 k^2 + b\bigr),\qquad k \text{ even},
\]
which is an integral Weierstrass model. For $a = 1$ and even $k$, the torsion-free conclusion is unconditional, that is $H_{k,1,b,1,1,n}(\Q)_{\mathrm{tors}} = \{\mathcal{O}\}$ (see Theorem~\ref{thm:general-3c-hyperelliptic}). For $a \neq 1$ it does not follow from the Diophantine correspondence; in fact, it fails to hold. For $a = 145$, $d = 3$, $c = m = n = s = q = 1$, $t = -11$, $k = -16$, we have $b = 145\cdot 8 - 3\cdot 121 = 797$. This choice of parameters satisfies the required hypotheses and we have the following elliptic curve
\[
  H_{-16,\,145,\,797,\,1,\,1,\,1}\colon\quad Y^2 = X^3 + 64\,X^2 - 22139325,
\]
whose Mordell--Weil group has nontrivial rational torsion: $H_{-16,145,797,1,1,1}(\Q)_{\mathrm{tors}} \cong \ZZ/2\ZZ$, generated by the $2$-torsion point $(261, 0)$. Thus $a = 1$ in Theorem~\ref{thm:general-3c-hyperelliptic} cannot be relaxed in general. This counterexample arises only in our setting with broader parametrisation; it is not a member of \cite{VaishyaSharma}, whose curves have constant term $-a^2(k^2 + b)$ with $b = 8a - 3$, where we have the specialisation $c = m = n = t = q = s = 1$, and $d=3$. Therefore, the family of elliptic curves is given by the Weierstrass equation $Y^2 = X^3 + \frac{1}{4} k^2 X^2 - a^2(k^2 + 8a - 3)$. A direct computation in \textsc{SageMath}~\cite{sagemath} found no nontrivial rational torsion among all admissible curves with $|a| \le 500$, even $|k| \le 60$ (yielding $2573$ distinct curves). This computational evidence is consistent with the claim in \cite[Theorem~1.2 (1)]{VaishyaSharma}.
\end{remark}

\begin{remark}\label{rem:hyperelliptic_counterpart}
We note that Theorem~\ref{thm:coprime_3} also admits a hyperelliptic counterpart, analogous to Theorem~\ref{thm:general-3c-hyperelliptic}, for the family of hyperelliptic curves $H_{k,a,b,c,m,n}$ over $\Q$. It can be shown that $H_{k,a,b,c,m,n}$ possesses no rational torsion points $(X,Y)$ satisfying the divisibility conditions specified in Lemma~\ref{lem:inverse-mapping}, alongside certain additional coprimality constraints arising from the insolvability conditions of Theorem~\ref{thm:coprime_3}.
\end{remark}

\bibliographystyle{plain}
\bibliography{references}
\end{document}